\documentclass{amsart}

\usepackage{tikz}

\usepackage[english]{babel}
\usepackage[T1]{fontenc}
\usepackage[utf8x]{inputenc}

\usepackage{mathrsfs}
\usepackage{amsfonts}
\usepackage{amssymb}
\usepackage{amsmath}

\usepackage{mathbbol}

\usepackage{color}
\usepackage{graphicx}

\usepackage{xspace}

\usetikzlibrary{decorations.pathmorphing}
\usetikzlibrary{decorations.markings}
\usepackage[pagewise,mathlines]{lineno}

\newcommand*\patchAmsMathEnvironmentForLineno[1]{%
  \expandafter\let\csname old#1\expandafter\endcsname\csname #1\endcsname
  \expandafter\let\csname oldend#1\expandafter\endcsname\csname end#1\endcsname
  \renewenvironment{#1}%
     {\linenomath\csname old#1\endcsname}%
     {\csname oldend#1\endcsname\endlinenomath}}%
\newcommand*\patchBothAmsMathEnvironmentsForLineno[1]{%
  \patchAmsMathEnvironmentForLineno{#1}%
  \patchAmsMathEnvironmentForLineno{#1*}}%
\AtBeginDocument{%
\patchBothAmsMathEnvironmentsForLineno{equation}%
\patchBothAmsMathEnvironmentsForLineno{align}%
\patchBothAmsMathEnvironmentsForLineno{flalign}%
\patchBothAmsMathEnvironmentsForLineno{alignat}%
\patchBothAmsMathEnvironmentsForLineno{gather}%
\patchBothAmsMathEnvironmentsForLineno{multline}%
}

\newcommand{\macro}[3]{\newcommand{#1}[#3]{#2}}
\macro{\lrwd}{\operatorname{lrwd}(#1)}{1}
\macro{\lbwd}{\operatorname{pwd}(#1)}{1}
\macro{\mat}{M_{#1}}{1}
\macro{\matind}{{#1}[{#2},{#3}]}{3}
\macro{\matgind}{\matind{\matg}{#1}{#2}}{2}
\macro{\subg}{#1\textrm{$\setminus$}#2}{2}
\macro{\field}{\mathbb{F}_{#1}}{1}
\macro{\angl}{\mathop\langle #1 \mathop\rangle}{1}
\macro{\Frwd}{\operatorname{rwd}^{{#1}}(#2)}{2}
\macro{\frwd}{\Frwd{\bF}{#1}}{1}

\def\matg{\mat{G}}
\def\ucutrk{\operatorname{cutrk}}
\def\rk{\operatorname{rk}}
\def\ie{\emph{i.e.}}

\def\restriction#1#2{\mathchoice
              {\setbox1\hbox{${\displaystyle #1}_{\scriptstyle #2}$}
              \restrictionaux{#1}{#2}}
              {\setbox1\hbox{${\textstyle #1}_{\scriptstyle #2}$}
              \restrictionaux{#1}{#2}}
              {\setbox1\hbox{${\scriptstyle #1}_{\scriptscriptstyle #2}$}
              \restrictionaux{#1}{#2}}
              {\setbox1\hbox{${\scriptscriptstyle #1}_{\scriptscriptstyle #2}$}
              \restrictionaux{#1}{#2}}}
\def\restrictionaux#1#2{{#1\,\smash{\vrule height .8\ht1 depth .85\dp1}}_{\,#2}}

\def\wrt{\emph{w.r.t.}\xspace}
\def\qo{\preceq}

\def\cM{\mathcal{M}}
\def\cN{\mathcal{N}}

\def\cD{\mathcal{D}}
\def\cI{\mathcal{I}}

\def\bN{\mathbb{N}}
\def\bF{\mathbb{F}}

\newtheorem{thm}{Theorem}
\newtheorem{lem}{Lemma}
\newtheorem{defn}{Definition}
\newtheorem{cor}{Corollary}
\newtheorem{prop}{Proposition}
\newtheorem{fact}{Fact}
\newtheorem{que}{Question}
\newtheorem{observation}{Observation}
\newtheorem*{thmmain}{Theorem \ref{thm:main}}

\title{An Upper Bound on the Size of Obstructions For Bounded Linear Rank-Width}

\author{Mamadou Moustapha Kant\'e \and O-joung Kwon}

\date{\today}

\address{Clermont-Universit{\'e}, Universit{\'e} Blaise Pascal, LIMOS, CNRS\\Complexe Scientifique des C{\'e}zeaux 63173 Aubi{\'e}re Cedex, France} \email{mamadou.kante@isima.fr} \address{Department
  of Mathematical Sciences, KAIST, 291 Daehak-ro Yuseong-gu Daejeon, 305-701 South Korea.}  \email{ojoung@kaist.ac.kr} \thanks{M.M. Kant{\'e} is supported by the French Agency for Research under the
  DORSO project.}
  \thanks{O. Kwon is supported by Basic Science Research
  Program through the National Research Foundation of Korea (NRF)
  funded by  the Ministry of Science, ICT \& Future Planning
  (2011-0011653).}

\begin{document}

\begin{abstract} 
  We provide a doubly exponential upper bound in $p$ on the size of forbidden pivot-minors for symmetric or skew-symmetric matrices over a fixed finite field $\bF$ of linear rank-width at most $p$.
  As a corollary, we obtain a doubly exponential upper bound in $p$ on the size of forbidden vertex-minors for graphs of linear rank-width at most $p$. This solves an open question raised by Jeong,
  Kwon, and Oum [\emph{Excluded vertex-minors for graphs of linear rank-width at most $k$}. European J. Combin., 41:242--257, 2014].  We also give a doubly exponential upper bound in $p$ on the size
  of forbidden minors for matroids representable over a fixed finite field of path-width at most $p$.

  Our basic tool is the pseudo-minor order used by Lagergren [\emph{Upper Bounds on the Size of Obstructions and Interwines}, Journal of Combinatorial Theory Series B, 73:7--40, 1998] to bound the
  size of forbidden graph minors for bounded path-width.  To adapt this notion into linear rank-width, it is necessary to well define partial pieces of graphs and merging operations that fit to
  pivot-minors.  Using the algebraic operations introduced by Courcelle and Kanté, and then extended to (skew-)symmetric matrices by Kanté and Rao, we define \emph{boundaried $s$-labelled graphs} and
  prove similar structure theorems for pivot-minor and linear rank-width.
\end{abstract}


\maketitle

\section{Introduction}\label{sec:1}

Rank-width is a graph width parameter, introduced by Oum and Seymour~\cite{OumS06}, generalizing tree-width in the sense that graphs of bounded tree-width have bounded rank-width.  Linear rank-width
is a variant of rank-width where the relationship between rank-width and linear rank-width is similar to that between tree-width and path-width.  Various properties of rank-width has been developed
recently, but the understanding of linear rank-width is still very restricted.  Vertex-minor and pivot-minor are the graph containment relations where rank-width and linear rank-width do not increase
when taking one of these operations. 

One way to understand the structure of a graph class is to identify the obstruction set.  Since graphs of (linear) rank-width at most $p$ are closed under taking vertex or pivot-minor one would know
whether the set of vertex or pivot-minors obstructions for (linear) rank-width $p$ can be constructed or at least described.  For instance, Oum proved that the obstructions for rank-width at most $p$
have sizes bounded by $(6^{p+1}-1)/5$ \cite{Oum05}, meaning that the obstruction set for rank-width at most $p$ has bounded size.  Therefore, one would wonder if this is still true for all graph
classes closed under taking vertex or pivot-minor.  While this question is still open, Oum~\cite{Oum08} proved that for every infinite sequence $G_1$, $G_2, \ldots$ of graphs of bounded rank-width, there
exist $i<j$ such that $G_i$ is isomorphic to a vertex-minor of $G_j$. A direct consequence is that every vertex or pivot-minor closed class of graphs of bounded rank-width has a finite set of
obstructions. Since rank-width is always less than or equal to linear rank-width and linear rank-width at most $p$ is closed under taking vertex or pivot-minor we can deduce the following as a
corollary.

\begin{cor}\label{cor:vertexminorwqo2} For fixed $p$, there exists a finite list of graphs $G_1, \ldots, G_m$ such that a graph has linear rank-width at most $p$ if and only if it does not have a
  vertex-minor isomorphic to $G_i$ for some $i\in \{1,2, \ldots, m\}$.
\end{cor}

A consequence of Corollary~\ref{cor:vertexminorwqo2} is, for fixed $p$, the existence of a Fixed Parameter Tractable (FPT for short) algorithm on $p$ that checks whether a graph has linear rank-width
at most $p$. However, from Corollary~\ref{cor:vertexminorwqo2} the algorithm is only existential because even though for a fixed graph $H$ there is an FPT algorithm on $p$ and the size of $H$ that
checks whether $H$ is a vertex-minor of a graph of rank-width $p$ \cite{CourcelleO07}, we do not know how to construct the set of obstructions.  Indeed, Corollary~\ref{cor:vertexminorwqo2}
does not tell how to identify all members of such a list, the cardinality of the list, or even the order of the largest graph of the list.  In this context, the following question is raised by Jeong,
Kwon, and Oum.

\begin{que}[\cite{JKO2014}]\label{que:que1}
For fixed $p$, find an explicit upper bound on the number of vertices in a forbidden vertex-minor for graphs of linear rank-width at most $p$.
\end{que}

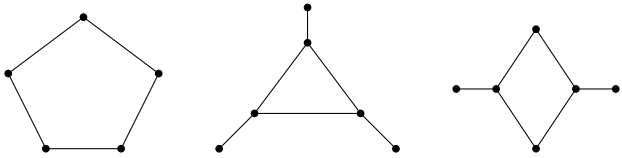
\begin{figure}[t]\centering
\tikzstyle{v}=[circle, draw, solid, fill=black, inner sep=0pt, minimum width=2.5pt]
\tikzset{photon/.style={decorate, decoration={snake}}}
\begin{tikzpicture}[scale=0.05]

\draw (30,55) -- (50,40) -- (40,20) -- (20,20) -- (10,40) -- (30,55);

\node [v] at (30,55) {};
\node [v] at (50,40) {};
\node [v] at (40,20) {};
\node [v] at (20,20) {};
\node [v] at (10,40) {};

\end{tikzpicture}\qquad
\begin{tikzpicture}[scale=0.047]

\draw (30,50) -- (30,60);
\draw (30,50) -- (15,30) -- (45,30) --(30,50);
\draw (15,30) -- (5,20);
\draw (45,30) -- (55,20);

\node [v] at (30,60) {};
\node [v] at (30,50) {};
\node [v] at (15,30) {};
\node [v] at (45,30) {};
\node [v] at (5,20) {};
\node [v] at (55,20) {};

\end{tikzpicture}\qquad
\begin{tikzpicture}[scale=0.053]

\draw (30,40) -- (40,25) -- (30,10) --(20,25) -- (30,40);
\draw (40,25) -- (50,25);
\draw (20,25) -- (10,25);

\node [v] at (30,40) {};
\node [v] at (40,25) {};
\node [v] at (30,10) {};
\node [v] at (20,25) {};
\node [v] at (50,25) {};
\node [v] at (10,25) {};

\end{tikzpicture}
\caption{Forbidden vertex-minors for graphs of linear rank-width $1$.}
\label{fig:lrw1}
\end{figure}

The case for $p=1$ is answered by Adler, Farley, and Proskurowski who gave in ~\cite{AdlerFP11} the complete list of forbidden vertex-minors for the class of graphs with linear rank-width at most $1$
(see Figure~\ref{fig:lrw1}).  For $p\ge 2$, Jeong, Kwon, and Oum~\cite{JKO2014} provided a general construction of forbidden vertex-minors for graphs of linear rank-width at most $p$, which shows that
the number of graphs in the list is at least doubly exponential in $p$.  Later, Adler, Kant\'{e}, and Kwon~\cite{AKK2014} established a way to construct all forbidden vertex-minors for linear
rank-width at most $p$ that are graphs of rank-width $1$. Nevertheless, there is no known result on the general upper bounds on the size of forbidden vertex-minors.


In this paper, we answer this question in a more general setting with matrices over a finite field and the pivot operation.  As usual, standard undirected graphs can be regarded as symmetric matrices
over the binary field, which represent the adjacencies of the graphs.  The notion of \emph{pivot complementation} in a graph, from which is based the notion of pivot-minor, originated from the
study of \emph{pivots} of matrices, sometimes called \emph{principal pivot transforms}~\cite{Tsatsomeros2000}.  Let $M$ be a $V\times V$ matrix over a field $\bF$ of the form
\[M:=\bordermatrix{
& S & V\setminus S\cr
S & A & B\cr
V\setminus S & C & D
}.
\] If $A=M[S]$ is nonsingular, then we define the \emph{pivot} at $S$ as the matrix
\[ 
M\ast S:=
\bordermatrix{
 & S & V\setminus S\cr
S & A^{-1} & A^{-1}B \cr
V\setminus S & -CA^{-1}  & D-CA^{-1}B 
}.
\]

One notices that the pivot operation preserves the (skew-) symmetricity of matrices, and also preserve the rank-width of matrices (definitions will be given in the next section), and in fact several
results concerning graph classes of bounded rank-width can be extended to (skew-)symmetric matrices of bounded rank-width. For instance Kanté and Rao \cite{KanteR13} proved that the obstructions for
(skew-) symmetric matrices of rank-width at most $p$ have sizes bounded by $(6^{p+1}-1)/5$, and Oum~\cite{Oum12} proved that (skew-) symmetric matrices of bounded rank-width are well-quasi-ordered by the pivot
operation. 


We will consider the more general notion of \emph{$\sigma$-symmetric matrices} over a finite field $\bF$, developed by Kant\'{e} and
Rao~\cite{KanteR13}, where $\sigma$, called \emph{sesqui-morphism} on $\bF$, is a bijective function on $\bF$ satisfying additional conditions. We can just observe that $\sigma$-symmetric matrices
generalize both symmetric and skew-symmetric matrices.  We call $G$ a \emph{$\sigma$-symmetric $\bF^*$-graph} if the adjacency matrix of it is a $\sigma$-symmetric matrix over the field $\bF$. 
  
The main theorem of this paper is the following.

\begin{thmmain}[Main Theorem]\label{thm:mainthm} Let $p$ be a positive integer, $\bF$ be a finite field of order $c$, and $\sigma$ be a sesqui-morphism on $\bF$. The number of vertices of every pivot-minor obstruction for  $\sigma$-symmetric
  $\bF^*$-graphs of linear rank-width at most $p$ is bounded by $c^{c^{O(p)}}$.
\end{thmmain}

For usual graphs, it is well-known that every pivot-minor of a graph is also a vertex-minor of it (see for instance \cite{Oum05}), and therefore, we obtain the following as a corollary.  (Notice that 
the notion of vertex-minor does not always exist for $\sigma$-symmetric $\bF^*$-graphs \cite{KanteR13}.)

\begin{cor}\label{cor:maincor1}
Let $p$ be a positive integer. The number of vertices of every  vertex-minor obstruction for  linear rank-width at most $p$ is bounded by $2^{2^{O(p)}}$.  
\end{cor}

Now, we can explicitly construct an FPT algorithm for linear rank-width using Corollary~\ref{cor:maincor1}. Moreover, if we slightly modify the result by Courcelle and Oum~\cite{CourcelleO07}, then for
a fixed $\sigma$-symmetric $\bF^*$-graph $H$ we can also show that there exists an FPT algorithm on $p$ and the size of $H$ to test whether a $\sigma$-symmetric $\bF^*$-graph $G$ of rank-width $p$
contains a pivot-minor isomorphic to $H$. Therefore, more strongly, we can explicitly construct an  FPT algorithm to test whether a given $\sigma$-symmetric $\bF^*$-graph
has linear rank-width at most $p$ or not using pivot-minors.

The second main corollary of Theorem \ref{thm:mainthm} is on the size of obstructions for representable matroids over a finite field of bounded path-width (definitions are given in
Section~\ref{sec:matroid}).
Geelen, Gerards, and Whittle~\cite{GeelenGW02} showed that for a fixed finite field $\bF$ and every infinite sequence $\cM_1$, $\cM_2, \ldots$ of $\bF$-representable matroids of bounded branch-width,
there exist $i<j$ such that $\cM_i$ is isomorphic to a minor of $\cM_j$. It implies that the class of $\bF$-representable matroids of path-width at most $p$ can be characterized by a finite list of
forbidden minors where $\bF$ is a finite field.  For a prime $q$, Kashyap~\cite{Kashyap2008} provided the forbidden minors for the $GF(q)$-representable matroids of path-width at most $1$, and a
partial set for path-width at most $2$.  Koutsonas, Thilikos, and Yamazaki~\cite{KTY2014} characterized the cycle matroids of outerplanar graphs with path-width at most $p$.  Our main theorem implies
the following on $\bF$-representable matroids for any finite field $\bF$.

\begin{cor}\label{cor:maincor2} 
  Let $p$ be a positive integer, and $\bF$ be a finite field of order $c$. If $\cM$ is an $\bF$-representable matroid and a minor obstruction for path-width at most $p$, then the size of the ground
  set of $\cM$ is bounded by $c^{c^{O(p)}}$.
\end{cor}

Hlin\v{e}n\'y~\cite{Petr2003,Petr2006} proved that for every positive integer $p$, every finite field $\bF$ of order $c$ and every fixed $\bF$-representable matroid $\cN$, there exists an FPT algorithm
on $(p,c,\cN)$ that checks whether a given $\bF$-representable matroid $\cM$ of branch-width $p$ and given with its representation, contains a minor isomorphic to $\cN$.  Using this algorithm, for
fixed $p$ and finite field $\bF$ of order $c$ we can explicitly construct an FPT algorithm on $p$ and $c$ that checks whether a given $\bF$-representable matroid $\cM$, given with its representation,
has path-width at most $p$.

\medskip

The main tools of this paper are the algebraic operations introduced by Courcelle and Kant\'{e}~\cite{CourcelleK09} and then generalised to $\sigma$-symmetric matrices by Kanté and Rao
\cite{KanteR13}, and the notion of pseudo-minor order used by Lagergren~\cite{Lagergren98} to obtain an upper bound on the size of minor obstructions for graphs of bounded path-width.  Similar to the
paper by Lagergren, we define a quasi-order $\lesssim$ on $\sigma$-symmetric $\bF^*$-graphs, called a \emph{pseudo-minor order}, such that

\begin{enumerate}
\item if $H$ is a pivot-minor of $G$, then $H \lesssim G$,~ 
\item  if $\lrwd{G\otimes H} \leq k$ and    $G'\lesssim G$, then $\lrwd{G'\otimes H} \leq k$.
\end{enumerate}
where $\lrwd G$ denotes the linear rank-width of $G$, and $G\otimes H$ denotes a kind of a sum of two $\sigma$-symmetric $\bF^*$-graphs.  For the sum of two $\sigma$-symmetric $\bF^*$-graphs, we will
use a labeling on the vertices, which has a similar role with the notion of \emph{boundary vertices} or \emph{terminal vertices} when we consider the clique sum in the graph minor theory.  
The proof consists of three parts.

\begin{enumerate}
\item We encode each linear layout of width $p$ of a forbidden pivot-minor $G$ in a compact way satisfying that 
  if the number of vertices in $G$ is at least $1 + \sum_{i=0}^{k}c(c+1)^i$, then we have a sequence of graphs $G_1, G_2, \ldots, G_{c+1}$ where $G_i$ is a proper pivot-minor of $G_{i+1}$ for each $1\le i\le c$.
\item We define a pseudo-minor order $\lesssim$ on $\sigma$-symmetric $\bF^*$-graphs. 
\item We prove that the length of the maximal chain with respect to $\lesssim$ is bounded by $(2p+1)\cdot c^{c^{2p}\cdot (6p +c)+ 2p^2+2p}$ where $c$ is the order of $\bF$
\end{enumerate}

Assuming the statement (1), together with the definition of pseudo-minor order, we may get an arbitrary long chain of graphs with respect to $\lesssim$ as we want, by increasing the size of vertices
in a forbidden pivot-minor.  However, it contradicts to the statement (3), and therefore, we conclude that the size of forbidden pivot-minor is bounded.

For $\bF$-representable matroids, we will establish a relation between $\bF$-representable matroids and skew-symmetric bipartite $\bF^*$-graphs, which can be seen as their fundamental graphs. Indeed,
we relate the minors of a $\bF$-representable matroid to the pivot-minors of its fundamental graph, and its path-width to the linear rank-width of its fundamental graph. We remark that
Oum~\cite{Oum05} already proved the same relations in the case of binary matroids, and our proof uses similar arguments. To our knowledge, this relation has not yet been noticed and we add the proof for
completeness. Therefore, Corollary~\ref{cor:maincor2} directly follows from the main theorem.


The paper is organized as follows. General notations, definitions and preliminary results are given in Section \ref{sec:preliminary}. We then adapt some results by Lagergren~\cite{Lagergren98} to our setting in Section \ref{sec:upper-bounds}. In Section \ref{sec:lrw} we prove that the number of vertices of an obstruction for linear rank-width at most $p$ is at most doubly
exponential in $O(p)$.  We conclude in Section \ref{sec:matroid} by proving that, for a fixed finite field $\bF$, the number of elements of any $\bF$-representable obstruction for path-width
at most $p$ is also at most doubly exponential in $O(p)$.

\section{Preliminaries}\label{sec:preliminary}

\subsection{General definitions}\label{subsec:gldfn}
The size of a set $A$ is denoted by $|A|$. For two sets $A$ and $B$, we let $A\setminus B:=\{x\in A\mid x\notin B\}$, and let $A\Delta B:=(X\setminus Y)\cup (Y\setminus X)$. The power-set of a set $V$
is denoted by $2^V$. We often write $x$ to denote the set $\{x\}$. We denote by $\bN$ the set containing zero and the positive integers, and by $[s]$ the set $\{1,\ldots,s\}$. 
For a finite
set $V$, we say that the function $f:2^V\to \bN$ is \emph{symmetric} if for any $X\subseteq V, ~f(X)=f(V\setminus X)$; $f$ is \emph{submodular} if for any $X,Y\subseteq V$,
$f(X\cup Y) + f(X\cap Y) \leq f(X) +f(Y)$.

Let $\bF$ be a finite field with characteristic $p$. We denote by $|\bF |$ the order of $\bF$, and let $\bF^*:= \bF\setminus \{0\}$. For $s\in \bN$, we denote by $\bF^s$ the set of vectors over $\bF$
of size $s$.  A set $X$ is called an \emph{$\bF$-multiset} if $X$ can have at most $(p-1)$ copies of each element. For an \emph{$\bF$-multiset} $X$, we define that 
\begin{align*}
  X\Delta_{\bF} \{x\}:=\begin{cases} 
  X\cup \{x\}  & \textrm{$X$ has at most $p-2$ copies of $x$ },\\
  X\setminus \{x, \ldots, x\} \text{ (remove all $x$)}  & \textrm{$X$ has $p-1$ copies of $x$}.
  \end{cases} 
\end{align*}

\subsection{Matrices}\label{subsec:matrices} For sets $R$ and $C$, an \emph{$(R,C)$-matrix} is a matrix where the rows are indexed by elements in $R$ and columns indexed by elements in $C$. For an $(R,C)$-matrix $M$, if
$X\subseteq R$ and $Y\subseteq C$, we let $\matind{M}{X}{Y}$ be the submatrix of $M$ where the rows and the columns are indexed by $X$ and $Y$ respectively. (If $X=Y$ we simply write $M[X]$, and if
one of $X$ or $Y$ is empty, by convention we let $\matind{M}{X}{Y}:=(0)$.) We let $\rk$ be the matrix rank-function (the field will be clear from the context). The \emph{order} of an $(R,C)$-matrix is
defined as $|R|\times |C|$. We often write $k\times \ell$-matrix to denote a matrix of order $k\times \ell$. We denote by $M^t$ the transpose of a matrix $M$.


Let $M$ be a matrix. A \emph{row operation on $M$} is a matrix obtained from $M$ by applying one of the following operations:~ (1) copy a row,~ (2) replace a row by a linear combination of rows. We
define similarly a \emph{column operation on $M$}. Given two matrices $M$ and $M'$, we write $M\preceq_r M'$ and $M\preceq_c M'$ whenever $M$ is a submatrix of a matrix obtained from $M'$ by row and
column operations respectively. We also write $M\cong M'$ if $M$ can be obtained from $M'$ by row and column operations and vice-versa. 
It is well-known that $\rk(M)\leq \rk(M')$ whenever $M\preceq_r M'$ or $M\preceq_c M'$, and hence $\rk(M)=\rk(M')$ if $M\cong M'$. 

We extend the matrix product as follows. Let $M$ and $N$ be two matrices. If the number of columns of $M$ equals the number of rows of $N$, then the product of $M$ and $N$ is as usual and we say that
it is \emph{well-defined}. Otherwise, we add some zero columns to $M$ (or zero rows to $N$) so that the product of the resulting matrices is well-defined.

\subsection{Linear Width}\label{subsec:lwidth} 

Let $f:2^V\to \bN$ be a symmetric function. A \emph{linear layout} of $f$ is an injective mapping $\pi:V\to [n]$. The \emph{width} of $\pi$ is $\max\limits_{1\leq i \leq n-1} \left\{f(\pi^{-1}([i]))\right\}$. The \emph{linear width} of $f$ is defined as
\begin{align*}
\min \{ \textrm{width of}\ \pi \mid \pi \ \textrm{is a
  linear layout of}\ f\}.
\end{align*}


Let $\pi:V\to [n]$ be a linear layout of an integer valued symmetric submodular function $f:2^V \to \bN$.  For two distinct $1\leq i < j \leq n-1$ we call $i$ and $j$ \emph{linked} if 
\begin{align*}
  \min\limits_{i\leq \ell \leq  j} f(\pi^{-1}([\ell])) = \min\limits_{\pi^{-1}([i]) \subseteq Z\subseteq V\setminus \pi^{-1}(\{j+1,\ldots,n\})} f(Z).
\end{align*}


A linear layout $\pi:V\to [n]$ of an integer valued symmetric submodular function $f:2^V\to \bN$ is said \emph{linked} if every two distinct $1\leq i < j \leq n-1$ are linked. The following is a straightforward
adaptation of the proof of \cite[Theorem 2.1]{GeelenGW02}.


\begin{thm}[{\cite[Theorem 2.1]{GeelenGW02}}]\label{thm:linked-layout} Every integer valued symmetric submodular function with linear width $k$ has a linked linear layout of width $k$.
\end{thm}

\begin{proof}  A careful analysis of the proof in \cite[Theorem 2.1]{GeelenGW02} shows that the given modification of the linear layout still produces a linear layout, and since the rest of the proof
  depends only on the fact that the function is an integer valued symmetric submodular one, we can conclude the statement. 
\end{proof}

Let $\pi:V\to [n]$ be a linear layout of an integer valued symmetric submodular function $f:2^V\to \bN$ and let $\lambda:[n-1] \to [p]$. We say that $i$ and $j>i$ are $\lambda$-linked if
$\lambda(\ell)\geq \lambda(i)=\lambda(j)$ for all $i\leq \ell \leq j$.

\subsection{Graphs} \label{subsec:graphs}

Our graph terminology is standard, see for instance \cite{Diestel05}. All graphs are finite, loop-free and undirected. The vertex set of a graph $G$ is denoted by $V_G$ and its edge set by $E_G$. We
will write $xy$ for an edge between $x$ and $y$ instead of $\{x,y\}$. For a graph $G$, we denote by $G[X]$, called the \emph{subgraph of $G$ induced by $X\subseteq V_G$}, the graph $(X,E_G\cap
(X\times X))$; we let $\subg{G}{X}$ be the subgraph $G[V_G\setminus X]$. Two graphs $G$ and $H$ are \emph{isomorphic} if there exists a bijection $h:V_G\to V_H$ such that $xy\in E_G$ if and only if
$h(x)h(y)\in E_H$.

For every undirected graph $G$, we let $\matg$ be its adjacency $(V_G,V_G)$-matrix over the binary field $\field{2}$ where $\matgind{x}{y}:=1$ if and only if $xy\in E_G$. The \emph{cut-rank} function
of every graph $G$ is the function $\ucutrk_G:2^{V_G}\to \bN$ where $\ucutrk_G(X) := \rk(\matgind{X}{V_G\backslash X})$. (By our convention on sub-matrices if one of $X$ and $V_G\setminus X$ is empty,
$\ucutrk_G(X)$ is equal to $0$.) This function is symmetric and submodular.  A \emph{linear layout} of a graph $G$ is a linear layout of $\ucutrk_G$, and the \emph{linear rank-width} of $G$, denoted by
$\lrwd{G}$, is the linear width of $\ucutrk_G$. One easily verifies that the linear rank-width of a graph is the maximum over the linear rank-width of its connected components (concatenate optimal
linear layouts of its connected components). Therefore, we will only deal from now on with connected graphs.

The cut-rank function was introduced by Oum and Seymour \cite{OumS06} and was the base for the definition of \emph{rank-width}, which is a good approximation for \emph{clique-width} (see
\cite{CourcelleO00} for the definition of clique-width). Rank-width is more interesting than clique-width and is actually related to a relation on undirected graphs, called \emph{vertex-minor}
\cite{Oum05}. Let us discuss about its trivial consequences on linear rank-width.

For a graph $G$ and a vertex $x$ of $G$, the \emph{local complementation at $x$} of $G$ consists in replacing the subgraph induced on the neighbors of $x$ by its complement. The resulting graph is
denoted by $G*x$.  For an edge $xy$ of $G$ we denote by $G\wedge xy$ the graph $G*x*y*x$. It is well-known that $G\wedge xy = G\wedge yx$ \cite{Oum05}. This latter operation is called \emph{pivot
  complementation}.  A graph $H$ is \emph{locally equivalent} (or \emph{pivot equivalent}) to a graph $G$ if $H$ can be obtained from $G$ be a sequence of local complementations (or
pivot-complementations); it is called a \emph{vertex-minor} (or \emph{pivot-minor}) of a graph $G$ if $H$ is isomorphic to an induced subgraph of a graph locally equivalent (or pivot equivalent) to
$G$, and it is a \emph{proper vertex-minor} (or \emph{proper pivot-minor}) if $V(H)\subset V(G)$. It is worth noticing that a pivot-minor is also a vertex-minor.

\begin{prop}[\cite{Oum05}] \label{prop:vm-lrw} Let $G$ be a graph and $x$ a vertex of $G$. For every $X\subseteq V(G)$, we have $\ucutrk_{G*x}(X) = \ucutrk_G(X)$. Hence, if $H$ is a vertex-minor (or a pivot-minor) of a
  graph $G$ then  $\lrwd{H}\leq \lrwd{G}$. 
\end{prop}

\subsection{$\sigma$-Symmetric $\bF^*$-Graphs and Pivot complementations}

Let $\bF$ be a field. An $\bF^*$-graph $G$ is a graph with an $\bF^*$-edge coloring $\ell$ of $G$ where $\ell:E_G\rightarrow \bF^*$ is the coloring function.  It is worth noticing that $\ell$ is not
necessarily symmetric, \ie, we may have $\ell(x,y)\ne \ell(y,x)$. 
For an $\bF^*$-graph $G$, we let $\matg$ be its adjacency $(V_G,V_G)$-matrix over the field $\bF$ where $\matgind{x}{y}:=\ell(xy)$ if $xy\in E_G$ and $\matgind{x}{y}:=0$ if otherwise.

Kant\'{e} and Rao~\cite{KanteR13} extended the notion of pivot complementations of usual undirected graphs into $\bF^*$-graphs having a certain property, called $\sigma$-symmetric. These $\sigma$-symmetric matrices generalize symmetric and skew-symmetric matrices.

For a field $\bF$, a bijective function $\sigma:\bF\rightarrow \bF$ is an \emph{involution} if $\sigma(\sigma(a))=a$ for all $a\in \bF$.
An involution $\sigma$ is called a \emph{sesqui-morphism} if the mapping that sends $x$ to $\sigma(x)/\sigma(1)$ is an automorphism.
An $(X,X)$-matrix $M$ over a field $\bF$ is \emph{$\sigma$-symmetric} if for every $x,y\in X$, $M[x,y]=\sigma(M[y,x])$.
A $\bF^*$-graph $G$ is \emph{$\sigma$-symmetric} if $M_G$ is $\sigma$-symmetric. Two $\sigma$-symmetric $\bF^*$-graphs $G$ and $H$ are \emph{simply isomorphic} if there is a bijection $h:V_G\to V_H$
such that for every $x$ and $y$ in $V_G$ we have that $M_G[x,y]=M_H[h(x),h(y)]$. 

We define pivot complementations on $\bF^*$-graphs.
Let $G$ be a $\sigma$-symmetric $\bF^*$-graphs, and let $x,y\in V_G$ such that $M_G[x,y]\neq 0$.
The pivot complementation at $xy$ in $G$ is the graph $G\wedge xy$ where 
$M_{G\wedge xy}[z,z]:=0$ for all $z\in V_G$, and for all $s,t\in V_G\setminus \{x,y\}$ where $s\neq t$, 
\begin{align*}
 M_{G\wedge xy}[s,t]:=M_G[s,t] - &\frac{M_G[s,x] \cdot M_G[y,t] }{M_G[y,x]} - \frac{M_G[s,y]\cdot M_G[x,t] }{M_G[x,y]},\\
 M_{G\wedge xy}[x,t]:=\frac{M_G[y,t]}{M_{G}[y,x]},  &\quad\quad M_{G\wedge xy}[y,t]:=\frac{\sigma(1) \cdot M_G[x,t]}{M_{G}[x,y]}, \\
 M_{G\wedge xy}[s,x]:=\frac{\sigma(1) \cdot M_G[s,y]}{M_{G}[x,y]},  &\quad\quad M_{G\wedge xy}[s,y]:=\frac{M_G[s,x]}{M_{G}[y,x]}, \\
 M_{G\wedge xy}[x,y]:=-\frac{1}{M_G[y,x]},  &\quad\quad M_{G\wedge xy}[y,x]:=-\frac{(\sigma(1))^2}{M_G[x,y]}.
\end{align*}
\begin{lem}[\cite{KanteR13}]\label{lem:preservesymmetric}
Let $G$ be a $\sigma$-symmetric $\bF^*$-graph and let $x,y\in V_G$ such that $M_G[x,y]\neq 0$.
Then $G\wedge xy$ is also $\sigma$-symmetric.
\end{lem}

The \emph{cut-rank} function
of every $\sigma$-symmetric $\bF^*$-graph $G$ is defined as the function $\ucutrk^{\bF}_G:2^{V_G}\to \bN$ where $\ucutrk^{\bF}_G(X) := \rk(\matgind{X}{V_G\backslash X})$ and the rank is computed over $\bF$.  This function is also symmetric and submodular~\cite{KanteR13}.  A \emph{linear layout} of a $\sigma$-symmetric $\bF^*$-graph $G$ is a linear layout of $\ucutrk^{\bF}_G$, and the \emph{linear rank-width} of $G$, denoted by
$\lrwd{G}$, is the linear width of $\ucutrk^{\bF}_G$.
If the field $\bF$ is clear from the context, we remove it from the notation.

\begin{lem}[\cite{KanteR13}]\label{lem:preserverank}
Let $G$ be a $\sigma$-symmetric $\bF^*$-graph and let $x,y\in V_G$ such that $M_G[x,y]\neq 0$. For every subset $X$ of $V_G$,
$\ucutrk^{\bF}_{G\wedge xy}(X) = \ucutrk^{\bF}_{G}(X)$.
\end{lem}

For an $(R,C)$-matrix $M=(m_{i,j})$ over a field $\bF$,
let $C_x$, $C_y$ be functions from $R\cup C$ to $\bF$, and $t\in \bF^{*}$, and $\sigma$ be a sesqui-morphism on $\bF$.
We define $M* (\sigma, C_x, C_y, t)$ 
as the matrix $M'=(m'_{i,j})$ where
\begin{align*}
  m'_{i,j}:= m_{i,j}- (\sigma(C_x(i))\cdot C_y(j))/\sigma(t) - (\sigma(C_y(i))\cdot C_x(j))/t. 
\end{align*}
If $\sigma$ is clear from the context, we remove it from the notation $M* (\sigma, C_x, C_y, t)$.  We remark that if $\bF$ is a finite field of characteristic $p$, then the matrix obtained from $M$ by
applying $p$ times the same operation $* (\sigma, C_x, C_y, t)$ is again $M$.  This operation is related to pivot complementations.  We also note that if $M$ is a $\sigma$-symmetric matrix over $\bF$,
then this operation preserves $\sigma$-symmetricity.

\begin{lem}\label{lem:preservesymmetric2}
Let $M$ be a $\sigma$-symmetric $(X,X)$-matrix over a field $\bF$ and let $C_x$, $C_y$ be functions from $X$ to $\bF$, and $t\in \bF^{*}$.
Then $M*(\sigma, C_x, C_y, t)$ is also $\sigma$-symmetric.
\end{lem}
\begin{proof}
We let $M'$ be the $(X\cup \{x,y\}, X\cup \{x,y\})$-matrix such that
\begin{enumerate}
\item  $M'[x,x]=M'[y,y]=0$,
\item $M'[x,y]=t$, $M'[y,x]=\sigma(t)$, and
\item for all $z\in X$, $M'[x,z]=C_x(z)$, $M'[z,x]=\sigma(C_x(z))$, 
$M'[y,z]=C_y(z)$, and $M'[z,y]=\sigma(C_y(z))$.
\end{enumerate}
If $G$ is the graph having $M'$ as the adjacency matrix, 
then $M_{G\wedge xy}[X]$ is exactly the same as $M*(\sigma, C_x, C_y, t)$.
By Lemma~\ref{lem:preservesymmetric}, $M*(\sigma, C_x, C_y, t)$ is $\sigma$-symmetric.
\end{proof}

\subsection{Linear encodings}

Let $\bF$ be a field and let $\sigma$ be a sesqui-morphism on $\bF$.
A \emph{linear encoding} of a $\sigma$-symmetric $\bF^*$-graph $G$ is a tuple $(N,P,M,L,t)$ where $t$ is an integer, $L:V_G\to [t]$ is an injective mapping, and for each $i\in [t-1]$ we require that $N(i)$, $P(i)$ and $M(i)$
are respectively $\ell_i\times n_i$, $\overline{\ell_i}\times p_i$ and $n_i\times p_i$-matrices over $\bF$ such that $\ell_i\leq 2^{n_i}$, $\overline{\ell_i}\leq 2^{p_i}$, and 
$\matgind{X_i}{\overline{X_i}} \cong N(i)\times M(i)\times P(i)^t$, where $X_i:=\{x\in V_G\mid L(x)\leq i\}$ and $\overline{X_i}:=V_G\setminus X_i$. The width is the maximum, over all $i \in [t-1]$,
of $\max\{l_i,p_i\}$. The following is proved implicitly in \cite{CourcelleK09}.


\begin{thm}[\cite{CourcelleK09}]\label{thm:algebra-lrw} Let $G$ be a $\sigma$-symmetric $\bF^*$-graph. For every linear layout $\pi:V_G\to [n]$ of width $k$ of $G$ one can construct a linear encoding
  $(N,P,M,L,n)$ of $G$ of width $k$ such that for each $i\in [n]$ $L(x_i):=\pi(x_i)$, and for each $i\in [n-1]$,
\begin{enumerate}
\item $M(i):=M_G[B_i, \overline{B_i}]$ where $B_i$ and $\overline{B_i}$ are indices\footnote{the smallest indices \wrt $\pi$.} of, respectively, row  and column basis of $M_G[X_i,\overline{X_i}]$.
\item For each $x\in X_i$, there is a row $u\in N(i)$ such that $M_G[x,\overline{X_i}] \preceq_c u\cdot M(i)\cdot P(i)^t$. Similarly, for each $y\in \overline{X_i}$, there is a row $v\in P(i)$ such that
  $M_G[X_i,y]\preceq_r N(i)\cdot M(i)\cdot v^t$. 
\item $N(i)$ and $P(i)$ have different row vectors. 
\end{enumerate}
 
Conversely, if $(N,P,M,L,t)$ is a linear encoding of width $k$ of $G$, then $L$ 
  is a linear layout of $G$ of width at most $k$.
 \end{thm}

\subsection{$s$-Labelled $\sigma$-Symmetric $\bF^*$-Graphs}\label{subsec:s-graphs}

Let $s\geq 0$ be an integer and $\bF$ be a field. An \emph{$s$-labelled $\bF^*$-graph} is a pair $(G,\gamma)$ where $G$ is an  $\bF^*$-graph and $\gamma:V_G\to \bF^s$ is a function such that the dimension of the vector space generated by
$\{\gamma(x)\mid x\in V_G\}$ has dimension $s$.  We denote by $\Gamma$ the matrix the rows of which are the vectors $\gamma(x)$ for $x\in V_G$, and $\Gamma[X]$ is $\Gamma$ restricted to $X\subseteq
V_G$.  
In addition, a \emph{boundaried $s$-labelled  $\bF^*$-graph} is a triple $(G, \gamma, \mu)$ 
	where $(G,\gamma)$ is an \emph{$s$-labelled  $\bF^*$-graph} 
 and $\mu$ is a $\bF$-multiset of triples $\{(v_i, v_j, t) \mid v_i, v_j\in \bF^s, t\in \bF^{*}\}$. 
 We call $\mu$ the boundary of $(G, \gamma, \mu)$. An $(s,\sigma,\bF^*)$-graph $(G, \gamma)$ is always regarded as a boundaried $(s, \sigma, \bF^{*})$-graph with the empty boundary.
 For a sesqui-morphism $\sigma$ on a field $\bF$, we shortly call as $(s, \sigma, \bF^{*})$-graph for an $s$-labelled $\sigma$-symmetric $\bF^{*}$-graph. Two $(s,\sigma,\bF^*)$-graphs
 $(G,\gamma_G,\mu_G)$ and $(H,\gamma_H,\mu_H)$ are
 \emph{simply isomorphic} if there is a simple isomorphism $h:V_G\to V_H$ between $G$ and $H$ such that $\mu_G=\mu_H$ and for every vertex $x\in V_G$ we have $\gamma_G(x)=\gamma_H(h(x))$.


For a pair of a boundaried $(s, \sigma, \bF^{*})$-graph $(G,\gamma_G, \mu_G)$ and an $(s, \sigma, \bF^{*})$-graph $(H,\gamma_H)$, and an $s\times s$-matrix $M$ over the field $\bF$, let $(G,\gamma_G, \mu_G)\otimes_{M} (H,\gamma_H)$ be the $s$-labelled $\bF^{*}$-graph $(K, \gamma_K)$ on the vertex set $V_G\cup V_H$ with the labelling $\gamma_G\cup \gamma_H$ such that
\begin{enumerate}
\item $M_K[V_G]:=M_G$,
\item for $v\in V_G, w\in V_H$, $M_K[v, w]:=\gamma_G(x)\cdot M\cdot \gamma_H(y)^t$, and $M_K[w,v]:=\sigma(M_k[v,w])$,
\item $M_K[V_H]:=M_H*(C_{v^1_1}, C_{v^1_2}, t_1)*(C_{v^2_1}, C_{v^2_2}, t_2)* \cdots *(C_{v^k_1}, C_{v^k_2}, t_k)$ where $\mu_G=\{(v^1_1, v^1_2, t_1), (v^2_1, v^2_2, t_2), \ldots, (v^k_1, v^k_2, t_k) \}$
and for each $1\le i\le k$, $C_{v^i_j}$ is a function from $V_H$ to $\bF$ that maps $y\in V_H$ to $v^i_j\cdot M\cdot \gamma_H(y)^t$.
 \end{enumerate}
 By Lemma~\ref{lem:preservesymmetric2}, $M_K[V_H]$ is again $\sigma$-symmetric, and therefore, $(K, \gamma_K)$ is again $\sigma$-symmetric.




Let $(G,\gamma_G, \mu_G)$ be a boundaried $(s, \sigma, \bF^{*})$-graph and let $xy\in E_G$ such that $M_G[x,y]=t\neq 0$. A \emph{pivot complementation at $xy$} of $(G,\gamma_G, \mu_G)$, denoted by
$(G,\gamma_G, \mu_G)\wedge xy$, is the boundaried $(s,\sigma,\bF^*)$-graph
$(G\wedge xy,\gamma, \mu)$ where $\mu:=\mu_G \Delta_{\bF} \{(\gamma(x), \gamma(y), t)\}$ and
\begin{align*}
  \gamma(z):=\begin{cases} 
  (1/\sigma(t))\cdot \gamma_G(y)  & \textrm{$z=x$},\\
  (\sigma(1)/t)\cdot \gamma_G(x)  & \textrm{$z=y$},\\
  \gamma_G(z) - (M_G[z,x]/\sigma(t))\cdot \gamma_G(y) - (M_G[z,y]/t)\cdot \gamma_G(x)  & \textrm{otherwise}. 
  \end{cases} 
\end{align*}
%
%
A deletion of a vertex $x$ from a boundaried $(s, \sigma, \bF^{*})$-graph $(G,\gamma_G, \mu_G)$, denoted by $(G,\gamma_G, \mu_G)\setminus x$, is the $(s, \sigma, \bF^{*})$-graph
$(G\setminus x,\gamma, \mu)$ where $\gamma$ is a restriction of $\gamma_G$ on $V(G)\setminus \{x\}$ and $\mu:=\mu_G$.  For $S\subseteq V(G)$, the \emph{induced subgraph} of $(G,\gamma_G, \mu_G)$ on
$S$ is obtained by removing the vertices in $V(G)\setminus S$.  A boundaried $(s, \sigma, \bF^{*})$-graph $(H,\gamma_H, \mu_H)$ is a \emph{pivot equivalent} of $(G,\gamma_G, \mu_G)$ if
$(H,\gamma_H, \mu_H)$ can be obtained from $(G,\gamma_G, \mu_G)$ by a sequence of pivot complementations. A boundaried $(s, \sigma, \bF^{*})$-graph $(H,\gamma_H, \mu_H)$ is a \emph{(proper)
  pivot-minor} of $(G,\gamma_G, \mu_G)$ if $(H,\gamma_H, \mu_H)$ is simply isomorphic to a (proper) induced subgraph of a boundaried $(s, \sigma, \bF^{*})$-graph pivot equivalent to
$(G,\gamma_G, \mu_G)$.

Let $L$ be a pivot-minor closed class of $\sigma$-symmetric $\bF^*$-graphs. A \emph{pseudo-minor order (pmo) for $L$} on the boundaried $(s, \sigma, \bF^{*})$-graphs is a quasi-order $\qo$ such that~
\begin{enumerate}
\item ($\qo$ respects $L$) if $(G,\gamma_G, \mu_G)$, $(H,\gamma_H, \mu_H)$ are boundaried $(s, \sigma, \bF^{*})$-graphs and $(K,\gamma_K)$ is an $(s, \sigma, \bF^{*})$-graph with $(H,\gamma_H,
\mu_H)\qo (G,\gamma_G, \mu_G)$ then $(G,\gamma_G, \mu_G)\otimes_{M} (K,\gamma_K)\in L$ implies $(H,\gamma_H, \mu_H) \otimes_{M} (K,\gamma_K)\in L$,~ 
\item ($\qo$ is a pmo) if $(H,\gamma_H, \mu_H)$
is a proper pivot-minor of $(G,\gamma_G, \mu_G)$, then $(H,\gamma_H, \mu_H) \qo (G,\gamma_G, \mu_G)$. 
\end{enumerate}
The \emph{length} of a pmo is its maximum chain length. The \emph{$p$-length} of $L$ is the
maximum over all $s\in [p]$ of the minimum length of a pmo for $L$ on the $(s, \sigma, \bF^{*})$-graphs.

\section{Bounds on the Length of Pseudo-Minor Orders}\label{sec:upper-bounds}


We prove that 
if $G$ is a pivot-minor obstruction for linear rank-width at most $p$, and $G$ is large, 
then we can find a sufficiently long strict chain of boundaried $(s, \sigma, \bF^{*})$-graphs $G_1, \ldots,$ $G_t$ with respect to the pivot-minor notion, \ie, $G_1$ is a proper pivot-minor of $G_2$, $G_2$ is a proper pivot-minor of $G_3$ and so on.
Instead of the property of bounded linear rank-width, we can generally prove it for the pivot-minor closed class $L$, but we need the condition that the obstruction has bounded linear rank-width. 
Using the notion of the $p$-length of a pmo, we can state it as follows.

\begin{thm}\label{thm:upper-bound} 
Let $\bF$ be a finite field and let $\sigma$ be a sesqui-morphism on $\bF$.
Let $L$ be a pivot-minor closed class of $\sigma$-symmetric $\bF^*$-graphs of $p$-length at most $c$. 
If $G$ is an obstruction for $L$ of linear rank-width at most
  $p$,  then the number of vertices of $G$ is bounded by $c^{O(p)}$.
\end{thm}

We fix  a finite field $\bF$ and  $\sigma$ a sesqui-morphism on $\bF$.
Most of the results in this section are generalizations of results in  \cite{Lagergren98,Oum05,Oum08},
 and we will also use Theorem
\ref{thm:linked-layout}. The following is trivial.

\begin{lem}\label{lem:trivial} Every $n$-vertex $\sigma$-symmetric $\bF^*$-graph of linear rank-width $k$ admits a linked linear layout $\pi:V(G)\to [n]$ of width $k$.
\end{lem}

Let us first recall some useful results. 

\begin{lem}[\cite{KanteR13}]\label{lem:submodular}
Let $G$ be a $\sigma$-symmetric $\bF^*$-graph and $v\in V(G)$ and let $w$ be an arbitrary neighbor of $v$ in $G$. 
If $(X_1, X_2)$ and $(Y_1, Y_2)$ are partitions of $V(G)\setminus v$, then
\[\ucutrk_{G\setminus v}(X_1) +\ucutrk_{G\wedge vw\setminus v}(Y_1)
\ge \ucutrk_G{(X_1\cup Y_1)} + \ucutrk_G(X_2\cap Y_2)-1.\]
\end{lem}

We prove an elementary version of the analogue of Tutte's linking theorem for $\sigma$-symmetric $\bF^*$-graphs. For usual graph cases, we refer to \cite[Theorem 6.1]{Oum05}.

\begin{thm}\label{thm:tutte} Let $G$ be a $\sigma$-symmetric $\bF^*$-graph and let $X$ and $Y$ be disjoint subsets of $V(G)$ such that $\ucutrk_G(Z)=k$ for $Z\in \{X, Y\}$. The following are
  equivalent.
\begin{enumerate}
\item $\min\limits_{X\subseteq Z \subseteq V(G)\setminus Y} \ucutrk_G(Z) \geq k$.
\item There exists a $\sigma$-symmetric $\bF^*$-graph $G'$ pivot equivalent to $G$ such that
\begin{align*} \ucutrk_{G'[X\cup Y]}(X)= k \end{align*}
and for each pair of subsets $A, B$ where $V(G)\setminus Y\subseteq A$ and $B\subseteq V(G)\setminus A$,
$\ucutrk_{G'[A\cup B]}(A)=\ucutrk_{G[A\cup B]}(A).$

\item There exists a sequence of pairs $(a_1, b_1), (a_2, b_2), \ldots, (a_m, b_m)$ which consist of vertices in $V(G)\setminus Y$ such that
\begin{align*} \ucutrk_{(G\wedge a_1b_1 \wedge a_2b_2 \cdots \wedge a_mb_m)[X\cup Y]}(X)= k. \end{align*}
\end{enumerate}
\end{thm}
\begin{proof} 
Clearly, (3) implies that (2) because all vertices $a_1, b_1, a_2, b_2, \ldots, a_m, b_m$ are contained in $V(G)\setminus Y$. For the part $((2) \Rightarrow (1))$,
	suppose that $G'$ is pivot equivalent to $G$.
	Then for all $Z$ satisfying $X\subseteq Z\subseteq V(G)\setminus Y$, 
	we have 
	\[k= \ucutrk_{G'[X\cup Y]}(X)\le \ucutrk_{G'}(Z)=\ucutrk_{G}(Z).\]

We show that (1) implies (3) by induction on $|V(G)\setminus (X\cup Y)|$. We may assume that $V(G)\setminus (X\cup Y)\neq \emptyset$.
	First suppose that for every vertex $v$ in $V(G)\setminus (X\cup Y)$,
	it has no neighbors on $V(G)\setminus Y$. 
	In this case, we can take an empty sequence because 
	$\ucutrk_{G[X\cup Y]}(X)=\ucutrk_{G}(X)=k$.
	So, we may assume that $V(G)\setminus (X\cup Y)$ contains a vertex $v$ where $v$ has a neighbor $w$ in $V(G)\setminus Y$.
		
	If \[\min\limits_{X\subseteq Z \subseteq V(G\setminus v)\setminus Y} \ucutrk_G(Z) \geq k,\]
	then by induction hypothesis, 
	there exists a sequence of pairs $(a_1, b_1)$, $(a_2, b_2)$, $\ldots,$ $(a_m, b_m)$ which consist of vertices in $V(G\setminus v)\setminus Y$ such that
\begin{align*} \ucutrk_{((G\setminus v)\wedge a_1b_1 \wedge a_2b_2 \cdots \wedge a_mb_m)[X\cup Y]}(X)= k. \end{align*}
	Since $(G\setminus v)\wedge a_1b_1 \wedge a_2b_2 \cdots \wedge a_mb_m=(G\wedge a_1b_1 \wedge a_2b_2 \cdots \wedge a_mb_m)\setminus v$, we have that \[\ucutrk_{(G\wedge a_1b_1 \wedge a_2b_2 \cdots \wedge a_mb_m)[X\cup Y]}(X)= k.\]
	So, we may assume that there exists a vertex set $Z_1$ such that 
$X\subseteq Z_1\subseteq V(G\setminus v)\setminus Y$ and $\ucutrk_{G\setminus v}(Z_1)\le k-1$. 
	By the same argument, we may also assume that there exists a vertex set $Z_2$ such that
$X\subseteq Z_2\subseteq V(G\wedge vw\setminus v)\setminus Y$ and $\ucutrk_{G\wedge vw\setminus v}(Z_2)\le k-1$.
By Lemma~\ref{lem:submodular}, either $\ucutrk_G(Z_1\cap Z_2)\le k-1$ or $\ucutrk_G(Z_1\cup Z_2)\le k-1$. Therefore, 
\[\min\limits_{X\subseteq Z \subseteq V(G)\setminus Y} \ucutrk_G(Z) \le k-1,\]
which is contradiction.	
\end{proof}



Given a linear layout $\pi:V(G)\to [n]$ of width $k$ of a $\sigma$-symmetric $\bF^*$-graph $G$ and for each $i\in [n-1]$, with Theorem \ref{thm:algebra-lrw}, one can associate with one boundaried
$(s,\sigma,\bF^*)$-graph $\alpha_i:=(G_i,\gamma_i, \mu_i=\emptyset)$ and a $(s,\sigma,\bF^*)$-graph $\beta_i:=(G'_i,\gamma'_i)$ and a matrix $M_i$ of order $s\times s$ such that $G=\alpha_i\otimes_{M_i}
\beta_i$. To be short we will call $(\alpha_i,\beta_i,M_i)$ a \emph{well-defined triplet}.  Using Theorem \ref{thm:tutte} we can state the following.

\begin{lem}\label{lem:right-encoding} Let $\pi:V(G)\to [n]$ be a linked linear layout of an $n$-vertex $\sigma$-symmetric $\bF^*$-graph $G$ of width $k$, and let $\lambda:[n-1]\to [k]$ be such that
  $\lambda(i):=\ucutrk_G(X_i)$. With every sequence $i_1<i_2 < \cdots < i_p$ of $p\geq 2$ indices such that $\lambda(i_j)=s$, and $i_j$ and $i_{j+1}$ are $\lambda$-linked, one can associate a
  $\sigma$-symmetric $\bF^*$-graph
  $G'$ pivot equivalent to $G$ such that $\ucutrk_{G'[X_{i_j}\cup \overline{X_{i_{j+1}}}]}(X_{i_j}) = s$ for every $j\in [p]$.
\end{lem}

\begin{proof} We prove it by induction on $p$. Assume first that $p=2$. By Theorem \ref{thm:tutte}  there is a graph $G'$ pivot equivalent to $G$ such that $\ucutrk_{G'[X_{i_1}\cup
    \overline{X_{i_2}}]}(X_{i_1}) = s$.  So we can conclude the statement. 

  Assume now that $p\ge 3$, and let $i_2< i_3 < \cdots < i_p$ be a sequence of $p-1$ indices that are $\lambda$-linked. By inductive hypothesis there is a graph $G'$ pivot equivalent  to $G$ such that
  $\ucutrk_{G'[X_{i_j}\cup \overline{X_{i_{j+1}}}]}(X_{i_j}) = s$ for every $2\leq j\leq p$. Since pivot complementations do not change the widths of the cuts
  $(X_i,\overline{X_i})$ the index $i_1$ is still $\lambda$-linked with the index $i_2$. 
  By Theorem \ref{thm:tutte}, 
  there exists a graph $G''$ pivot equivalent to $G'$ such that
\begin{align*} \ucutrk_{G''[X_{i_1}\cup \overline{X_{i_2}}]}(X_{i_1})= s \end{align*}
and 
for each pair of subsets $A, B$ where $V(G')\setminus \overline{X_{i_2}}\subseteq A$ and $B\subseteq V(G')\setminus A$,
$\ucutrk_{G''[A\cup B]}(A)=\ucutrk_{G'[A\cup B]}(A).$
For every $2\leq j\leq p$, since $V(G')\setminus \overline{X_{i_2}} \subseteq X_{i_j}$ and $\overline{X_{i_{j+1}}}\subseteq V(G')\setminus X_{i_j}$, from the second statement, we have that
  \begin{align*}
  &\ucutrk_{G''[X_{i_j}\cup \overline{X_{i_{j+1}}}]}(X_{i_j}) =\ucutrk_{G'[X_{i_j}\cup \overline{X_{i_{j+1}}}]}(X_{i_j})= s. 
  \end{align*} Therefore, we conclude the result.
 \end{proof}

 The following now follows the same proof line as in \cite[Section 3]{Lagergren98}. Let $L$ be a pivot-minor closed class of $\sigma$-symmetric $\bF^*$-graphs and let $G$ be an $n$-vertex obstruction
 for $L$ of linear rank-width at most $k$, which is a $\sigma$-symmetric $\bF^*$-graph. We moreover assume that we are given a fixed linked linear layout $\pi:V(G)\to [n]$ of width $k$ of $G$, and $\lambda:[n-1]\to [k]$ such that
 $\lambda(i):=\ucutrk_G(X_i)$.  Let
\begin{align*}
  l_k(c)&:=1 + \sum_{i=0}^{k}c(c+1)^i.
\end{align*}

\begin{lem}\label{lem:reduce-obstructions} If $n\geq l_k(c)$ then for some $0\leq s \leq k$ there is a sequence $S:=(i_1,i_2,\ldots,i_{c+1})$ of $(c+1)$ indices such that $i_1<i_2<\cdots <i_{c+1}$,
  $\lambda(i_j)=s$ for all $1\leq j\leq c+1$ and any two consecutive indices $i_j < i_\ell$ are $\lambda$-linked.
\end{lem}

\begin{proof} The proof is similar to the one in \cite[Lemma 3.2]{Lagergren98}, we include it for completeness. Let $\ell\geq 0$ be the greatest number such that there is a sub-interval $I$ of $[n]$ of
  length $\geq 1+\sum_{i=\ell}^{k} c(c+1) ^{i-\ell}$ where for each $i\in I$ we have $\lambda(i)\geq \ell$. Notice that such an integer $\ell$ and interval $I$ exist because $n\geq
  1+\sum_{j=0}^{k} c(c+1)^j$. 

  Assume first that $\ell=k$. Then, $\sum_{j=\ell}^{k} c(c+1)^{j-\ell} = c$, \ie, $I$ has length at least $c+1$ and for each $i\in I$ we have $\lambda(i)=k$.  Therefore, we can choose in $I$ a
  sequence $S$ of length $c+1$ as stated in the lemma.

  Assume now that $\ell < k$, and let $S:=\{j\in I\mid \lambda(j)=\ell\}$. If the lemma is false then $|S|\leq c$.  Therefore, there are at most $c+1$ sub-intervals of $I$ without an index $j\in S$. 
  At least one such sub-interval $I'$ should have length at least 
  \begin{align*}
    \frac{\sum_{i=\ell}^{k} c(c+1)^{i-\ell} -c}{c+1} & = \frac{\sum_{i=\ell+1}^{k} c(c+1)^{i-\ell}}{c+1} \\ & = \sum_{i=\ell+1}^{k} c(c+1)^{i-(\ell+1)}.
  \end{align*}
  Since for each $j\in I'$ we have $\lambda(j)\geq \ell+1$, we contradict the  choice of $\ell$ to be the maximum. Hence, $|S|\geq c+1$, and again we can choose $c+1$ indices as stated in the lemma. 
\end{proof}

\begin{lem}\label{lem:size-layout} If $L$ has $k$-length at most $c$, then $n <  l_k(c)$. 
\end{lem}

\begin{proof} The proof is the same as in \cite[Lemma 3.3]{Lagergren98}, and again we include it for completeness. Assume that $n\geq l_k(c)$. By Lemma \ref{lem:reduce-obstructions} for some $0\leq
  s\leq k$ there is a sequence $S:=(i_1,i_2,\ldots,i_{c+1})$ of $c+1$ indices such that $i_1< i_2 < \cdots < i_c < i_{c+1}$, $\lambda(i_j)=s$ and $i_j$ and $i_{j+1}$ are $\lambda$-linked. Let
  $\lesssim$ be a pmo with $k$-length at most $c$ for $L$ on the $s$-labelled graphs. By Lemma \ref{lem:right-encoding} there exists a graph $G'$ pivot equivalent to $G$  such that $\ucutrk_{G'[X_{i_j}\cup \overline{X_{i_{j+1}}}]}(X_{i_j}) = s$ for every $j\in [c+1]$. Let $(\alpha_{i_{c+1}},\beta_{i_{c+1}},M_{i_{c+1}})$ be the well-defined triplet at $i_{c+1}$
  associated with $\pi$ and $G'$, and for every $i_j<i_{c+1}$ let $\alpha_{i_j}$ be the subgraph of $\alpha_{i_{c+1}}$ induced by $X_{i_j}$. Since whenever $i_j < i_\ell$ we have $\alpha_{i_j}$ is
  an induced subgraph of $\alpha_{i_\ell}$, we can then conclude that $\alpha_{i_j} \lesssim \alpha_{i_\ell}$ because $\lesssim$ is a pmo. 

  Since the pmo $\lesssim$ has $k$-length at most $c$, there are $i_j$ and $i_\ell$ in $S$ such that $i_j < i_\ell$ and $\alpha_{i_\ell}\lesssim \alpha_{i_j}$. Let us choose $i_j$ and
  $i_\ell$ to be the greatest indices with the property that  $\alpha_{i_\ell} \lesssim \alpha_{i_j}$. We can deduce then that $i_{\ell}=i_{j+1}$, otherwise since $\alpha_{i_j}$ is a proper
  pivot-minor of $\alpha_{i_{j+1}}$ we would also have by transitivity $\alpha_{i_\ell} \lesssim \alpha_{i_{j+1}}$ contradicting $(i_j,i_\ell)$ are the greatest indices.

  Since $\lambda(X_{i_{j+1}})=s$ and $\ucutrk_{G'[X_{i_{j+1}}\cup \overline{X_{i_{j+2}}}]}(X_{i_j}) = s$, one can deduce from Theorem \ref{thm:algebra-lrw} that there exist an $s$-labelled graph
  $\beta_{i_{j+1}}$ and an $s\times s$-matrix $M_{i_{j+1}}$ such that $(\alpha_{i_{j+1}},\beta_{i_{j+1}},M_{i_{j+1}})$ is a well-defined triplet at $i_{j+1}$ associated with $\pi$ and $G'$. From
  Proposition \ref{prop:vm-lrw} we know that $G'$ is not in $L$.  Now $\alpha_{i_j} \otimes_{M_{i_{j+1}}} \beta_{i_{j+1}}$ is not in $L$ otherwise $G'=\alpha_{i_{j+1}} \otimes_{M_{i_{j+1}}}
  \beta_{i_{j+1}}$ would be in $L$ because $\alpha_{i_{j+1}} \lesssim \alpha_{i_j}$ and $\lesssim$ is a pmo.  Now, $\alpha_{i_j}\otimes_{M_{i_{j+1}}} \beta_{i_{j+1}}$ is a proper induced subgraph of
  $G'$, \ie, $G$ has a proper pivot-minor not in $L$. This contradicts the fact that $G$ is an obstruction for $L$, and then we conclude that that $n<l_k(c)$.
\end{proof}

\begin{proof}[Proof of Theorem~\ref{thm:upper-bound}] By  Lemma \ref{lem:trivial} $G$ has a linked linear layout $\pi:V(G)\to [|V(G)|]$ of width at most $p$. By Lemma \ref{lem:size-layout} $|V(G)|\leq c^{O(p)}$. 
\end{proof}

\section{Obstructions for Linear Rank-Width}\label{sec:lrw}


In this section, we prove the main result of this paper.

\begin{thm}[Main Theorem]\label{thm:main} 
Let $\bF$ be a finite field and let $\sigma$ be a sesqui-morphism on $\bF$.
If $G$ is a pivot-minor obstruction for $\sigma$-symmetric $\bF^*$-graphs of linear rank-width at most $p$, then $|V_G|$ is at most doubly exponential in $O(p)$.
\end{thm}

To prove this theorem, we first construct a pseudo-minor order in terms of some systems, called \emph{linear $s$-profiles}, which can be obtained from linear layouts by extracting essential sets of vectors.
In the second phase, we show that the $p$-length of this particular pseudo-minor order is bounded, by proving that the number of all possible minimal linear $s$-profiles is bounded.

\subsection{Constructing a Proper Pseudo-Minor Order} 

Let $s\geq 0$. We fix a finite field $\bF$and $\sigma$ a sesqui-morphism on $\bF$. A \emph{linear $s$-profile} is a tuple $(Y:=(Y_1,Y_2),Z:=(Z_1,Z_2),\mu, M,t)$ where $t$ is an integer, and for every $i\in [t]$, $M(i)$, $Y(i):=\left(Y_1(i)\ \
  Y_2(i)\right)$ and $Z(i):=\left(Z_1(i)\ \ Z_2(i)\right)$ are matrices over $\bF$ such that the rows of $Y_2(i)$ and $Z_2(i)$ are vectors in $\bF^s$, $\mu$ is a $\bF$-multiset of triples $\{(v_i,
v_j, t) \mid v_i, v_j\in \bF^s, t\in \bF^{*}\}$, and $Y_1(i)\cdot M(i)\cdot Z_1(i)^t$ is well-defined. We moreover require that for each $1\leq i \leq t$ the matrix $Rest(Y_2(i))$ is always a sub-matrix of
$Rest(Y_2(j))$ for all $j>i$, and similarly $Rest(Z_2(j))$ is a sub-matrix of $Rest(Z_2(i))$ for all $i<j$, where $Rest(A)$ is the matrix restricted to non-repeated row vectors. 

We now define widths of linear $s$-profiles. An \emph{$(s,p)$-matrix tuple} is a tuple $\cD:=(\Gamma, N, P:=(P_1,P_2), Q:=(Q_1,Q_2))$ where $\Gamma$ is of order $s\times s$, $P_2$ and $Q_2$ of order
at most $|\bF |^{p+s}\times s$, $P_1$ and $Q_1$ of order at most $|\bF |^{p+s}\times p$ and $N$ of order at most $p\times p$. 
The row indices of $P$ and $Q$ are denoted by $V(P)$ and $V(Q)$, respectively.
Let $E:=(Y,Z,\mu, M,t)$ be a linear $s$-profile with $\mu=\{(v^1_1, v^1_2, t_1), (v^2_1, v^2_2, t_2), \ldots, (v^k_1, v^k_2, t_k) \}$.
For each $i\in [t]$ and an $(s,p)$-matrix
tuple $\cD$ with $p\geq \max\{\rk(Y_1(i)\cdot M(i)\cdot Z_1(i)^t\mid i\in [t]\}$ let
\begin{align*}
  A_{E,\cD}(i)&:=\begin{pmatrix} Y_1(i)\cdot M(i)\cdot Z_1(i)^t & Y_2(i)\cdot \Gamma \cdot {Q_2}^t \\ (Z_2(i)\cdot \Gamma\cdot
      {P_2}^t)^t & P' \end{pmatrix}.
\end{align*}
where $P'=(P_1\cdot N\cdot Q_1^t)*(C_{v^1_1}, C_{v^1_2}, t_1)*(C_{v^2_1}, C_{v^2_2}, t_2)* \cdots *(C_{v^k_1}, C_{v^k_2}, t_k)$ where for each $1\le i\le k$, 
\begin{enumerate}
\item $C_{v^i_1}$ is the function from $V(P)$ to $\bF$ that maps $y\in V(P)$ to $v^i_1\cdot \Gamma\cdot P_2(y)^t$,
\item $C_{v^i_2}$ is the function from $V(Q)$ to $\bF$ that maps $y\in V(Q)$ to $v^i_2\cdot \Gamma\cdot Q_2(y)^t$.
\end{enumerate}


For each $i$, let 
\begin{align*} 
  p-wd(i)&:=\max\limits_{\textrm{over all $(s,p)$-matrix tuple $\cD$}}\{\rk(A_{E,\cD}(i))\}.
\end{align*}  
The \emph{$p$-width of $E$} is defined as $\max\{p-wd(i)\mid i\in [t]\}$. It is worth noticing that the $p$-width of a linear $s$-profile is always a finite integer.

\begin{fact}\label{fact:dual} If $(Y,Z,\mu,M,t)$ is a linear $s$-profile of $p$-width $k$, then $(Y',Z',\mu, M',t)$ is also a linear $s$-profile of $p$-width $k$ with $Y'(i):=Z(t-i+1)$, $Z'(i):=Y(t-i+1)$,
  and $M'(i):=M(t-i+1)^t$, called the \emph{dual} of $E$ and denoted by $E^d$. Furthermore, for every $(s,p)$-matrix tuple $\cD:=(\Gamma, N, P:=(P_1,P_2), Q:=(Q_1,Q_2))$ we have $\rk(A_{E,\cD}(t-i+1)) =
  \rk(A_{E^d,\cD^d}(i))$ where $\cD^d:=(\Gamma, N^t, (Q_1,Q_2), (P_1,P_2))$.
\end{fact}


Let $(G,\gamma, \mu)$ be a boundaried $(s, \sigma, \bF^{*})$-graph
 and $(N,P,M,L,t)$ a linear encoding of $G$ of width $k$ constructed from Theorem \ref{thm:algebra-lrw}. Recall that
an $(s,\sigma,\bF^*)$-graph $(G, \gamma)$ is regarded as a boundaried $(s, \sigma, \bF^{*})$-graph with the empty boundary.
  Then the tuple $(Y,Z,\mu, M,t)$ where $Y(t):=(0\ \gamma)$, $Z(t):=0$, $M(t):=0$, and for each $i\in [t-1]$ we have the matrix $Y(i)$ as a set of rows (ordered following the order of $N$ and $\gamma$)
\begin{align*} 
 & \{(u\ \ v)\mid \exists\ x\in X_i,\ u\in N(i),\ \gamma(x)=v,\ M_G[x,\overline{X_i}]=u\cdot M(i)\cdot P(i)^t\}\\ \intertext{and similarly the matrix $Z(i)$ as a set of rows} & \{(u\ \
 v)\mid \exists\ y\in \overline{X_i},\ u\in P(i),\ \gamma(y)=v,\ M_G[X_i,y]=N(i)\cdot M(i)\cdot u^t\},
\end{align*}
is a linear $s$-profile of $k$-width at most $|\bF |^{2\cdot(k+s)}$ called the \emph{$(G,\gamma, \mu)$-profile of $(N,P,M,L,t)$}. A linear $s$-profile $E$ of $k$-width at most $|\bF |^{2\cdot (k+s)}$ is a
\emph{linear $s$-profile of $(G,\gamma, \mu)$} if $E$ is a $(G,\gamma, \mu)$-profile of some linear encoding of $G$ of width $\leq k$.


\begin{defn}\label{defn:subdivision} Let $(Y,Z,\mu, M,t)$ be a linear $s$-profile and $i\in [t]$. Then $(Y',Z',\mu, M',t+1)$ is a \emph{subdivision of $(Y,Z,\mu, M,t)$ at $i$} if the following conditions are satisfied 
  \begin{itemize}
  \item $Y'(j):=Y(j)$, $Z'(j):=Z(j)$, $M'(j):=M(j)$ for all $j\leq i$,
  \item $Y'(j+1):=Y(j)$, $Z'(j+1):=Z(j)$, $M'(j+1):=M(j)$ for all $i\leq j \leq t$.
  \end{itemize}
  If there is a sequence $E_1,\ldots,E_r$ of linear $s$-profiles such that $E_{i+1}$ is a subdivision of $E_i$, then we also call $E_r$ a \emph{subdivision} of $E_1$. 
\end{defn}

It is worth noticing that if $E$ is a linear $s$-profile of $p$-width $k$ of $(G,\gamma, \mu)$, then any subdivision of $E$ is also a linear $s$-profile of $p$-width $k$ of $(G,\gamma, \mu)$.  



\begin{defn}\label{defn:dominance} 
  A linear $s$-profile $E:=(Y,Z,\mu, M,t)$ is \emph{directly $p$-dominated} by another linear $s$-profile $E':=(Y',Z',\mu', M',t)$, written $E\leq_{DD}^p E'$, if for each $i\in [t]$ and each $(s,p)$-matrix tuple
  $\cD:=(\Gamma, N, P:=(P_1,P_2), Q:=(Q_1,Q_2))$ we have \begin{align*} \rk \big(A_{E,\cD}(i)\big) \leq \rk\big(A_{E',\cD}(i)\big). \end{align*}


One can easily check that the relation $\leq_{DD}^p$ is transitive and since it is reflexive is a quasi-order. A linear $s$-profile $E_1$ is \emph{$p$-dominated} by another linear $s$-profile
$E_2$, written $E_1\leq_D^p E_2$, if there are subdivisions $E_1'$ and $E_2'$ of respectively $E_1$ and $E_2$ such that $E_1' \leq_{DD}^p E_2'$.
\end{defn}

\begin{fact}[\cite{Lagergren98}]\label{fact:dd-subdivision} If $E'$ is a subdivision of a linear $s$-profile $E$ at $i$, and $E\leq_{DD}^p F$ 
  then there exists a subdivision $F'$ of $F$ at $i$ such that $E'\leq_{DD}^p F'$. Therefore, if $E_1,\ldots, E_r$ is a sequence of linear $s$-profiles such that $E_{i+1}$ is a subdivision of $E_i$ at
  $i_j$, and $E_1\leq_{DD}^p F_1$, then there exists a sequence $F_1,\ldots, F_r$ of linear $s$-profiles such that $F_{i+1}$ is a subdivision of $F_i$ at $i_j$ and $E_{i+1} \leq_{DD}^p F_{i+1}$.
\end{fact}

\begin{lem}[\cite{Lagergren98}]\label{lem:common-subdivision} If $E'$ and $E''$ are two subdivisions of a linear $s$-profile $E$, then there exists a common subdivision $F$ of $E'$ and $E''$. 
\end{lem}

\begin{prop}{\cite[Theorem 4.3]{Lagergren98}}\label{prop:dominance-quasiorder} The relation $\leq_D^p$ is a quasi-order. 
\end{prop}

We denote by $\simeq_D^p$ the relation such that $E\simeq_D^p F$ if $E\leq_D^p F$ and $F\leq_D^p E$. It is clearly an equivalence relation from Proposition \ref{prop:dominance-quasiorder}. 

\begin{observation}\label{obs:rk} Let
$E:=(Y,Z,\mu, M,t)$ be a linear $s$-profile. 
We first observe that if $\rk(Y_1(i)\cdot M(i)\cdot Z_1(i)^t)=k$ then we can respectively replace $Y(i)$, $Z(i)$ and $M(i)$ by some $Y'(i)$, $Z'(i)$ and
$M'(i)$, each of $Y'(i)$ and $Z'(i)$ of order at most $|\bF |^{k+s}\times (k+s)$, $M'(i)$ of order $k\times k$, and obtain a linear $s$-profile equivalent to $E$ \wrt $\simeq_D^p$. Moreover, the linear
$s$-profile $E':=(Y',Z',\mu,M',t)$ obtained from $E$ by adding to $Y(i)$, $Z(i)$ and $M(i)$ some zero rows and zero columns for some $i\in [t]$ is equivalent to $E$ \wrt $\simeq_D^p$.  We can therefore
assume that if $E:=(Y,Z,\mu,M,t)$ is a linear $s$-profile, then for each $i\ne j$, $Y(i)$ and $Y(j)$ have the same number of rows and columns, and similarly for $Z(i)$ and $Z(j)$, and $M(i)$ and $M(j)$,
and if $k:=\max\{\rk(Y_1(i)\cdot M(i)\cdot Z_1(i)^t)\mid i\in [t]\}$, the $M(i)$s are of order $k\times k$, and the $Y(i)$ and $Z(i)$s are of order at most $|\bF |^{k+s}\times (k+s)$. 
\end{observation}

Let $p\geq 0$ be a positive integer. For a boundaried $(s, \sigma, \bF^{*})$-graph $(G,\gamma, \mu)$ we denote by $Ext_p(G,\gamma, \mu)$ the set of all subdivisions of all its linear $s$-profiles of $p$-width at most $p$. Let $\lesssim^p$ be the
relation such that for any two boundaried $(s, \sigma, \bF^{*})$-graphs $(G,\gamma_G, \mu_G)$ and $(H,\gamma_H, \mu_H)$ we have $(H,\gamma_H, \mu_H) \lesssim^p (G,\gamma_G, \mu_G)$ if for every $F\in Ext_p(G,\gamma_G, \mu_G)$ there is $E\in
Ext_p(H,\gamma_H, \mu_H)$ such that $E\leq_{DD}^p F$, and if $(N_G,P_G,M_G,L_G,t)$ and $(N_H,P_H,M_H,L_H,t)$ are linear encodings associated respectively with $E$ and $F$, then for each $1\leq i \leq t$ we
have $|L_H^{-1}(i)|\leq |L_G^{-1}(i)|$. We want to prove that $\lesssim^p$ is a pmo for graphs of linear rank-width at most $p$ on $(s, \sigma, \bF^{*})$-graphs.

As in \cite{Lagergren98} let us introduce a notion of mergeability. Let $\Gamma$ be a matrix and let $(N,P,Q,L,t)$ and $(N',P',Q',L',t)$ be two linear encodings of $G$ and $H$
respectively, and let $E:=(Y,Z,\mu_G, M,t)$ and $E':=(Y',Z',\mu_H, M',t)$ be $(G,\gamma_G, \mu_G)$ and $(H,\gamma_H, \mu_H)$ profiles of $(N,P,Q,L,t)$ and $(N',P',Q',L',t)$ respectively,
where $\mu_G=\{(v^1_1, v^1_2, t_1), (v^2_1, v^2_2, t_2), \ldots, (v^k_1, v^k_2, t_k) \}$ and $\mu_H=\emptyset$. 
The row indices of $Y'$ and $Z'$ are denoted by $V(Y')$ and $V(Z')$, respectively.

We say that $E$ is \emph{$p$-mergeable with} $E'$ \emph{by $\Gamma$} if for every $i\in [t]$
\begin{align*}
  \rk \begin{pmatrix} Y_1(i)\cdot M(i) \cdot Z_1(i)^t & & Y_2(i)\cdot \Gamma \cdot {Z'}_2(i)^t \\ \\ \left(Z_2(i)\cdot \Gamma \cdot {Y'}_2(i)^t\right)^t & & N \end{pmatrix} \leq p.
\end{align*}
where $N=({Y'}_1(i) \cdot M'(i) \cdot {Z'}_1(i)^t)*(C_{v^1_1}, C_{v^1_2}, t_1)*(C_{v^2_1}, C_{v^2_2}, t_2)* \cdots *(C_{v^k_1}, C_{v^k_2}, t_k)$
and for each $1\le i\le k$, 
\begin{enumerate}
\item $C_{v^i_1}$ is a function from $V(Y')$ to $\bF$ that maps $y\in V(Y')$ to $v^i_1\cdot \Gamma\cdot Y_2'(y)^t$,
\item $C_{v^i_2}$ is a function from $V(Z')$ to $\bF$ that maps $y\in V(Z')$ to $v^i_2\cdot \Gamma\cdot Z_2'(y)^t$.
\end{enumerate}

The following is a direct consequence of the definitions of direct $p$-dominance and $p$-mergeability. 


\begin{fact}\label{fact:dd-merge} Let $E,E'$ and $E''$ be linear $s$-profiles, and $\Gamma$ a matrix. If $E'\leq_{DD}^p E$, and $E$ is $p$-mergeable with $F$ by $\Gamma$, then $E'$ is $p$-mergeable with $F$ by $\Gamma$.
\end{fact}

\begin{lem}\label{lem:merge-lrw} Let $E$ and $F$ be linear $s$-profiles of respectively $(G,\gamma_G, \mu_G)$ and $(H,\gamma_H)$, and let $(N,P,M,L,t)$ and $(N',P',M',L',t)$ be  linear encodings of
  $G$ and $H$ associated with $E$ and $F$ respectively. If for each $1\leq i \leq t$ at most one vertex of $V_G\cup V_H$ is mapped into $i$ by $L\cup L'$, and $E$ is $p$-mergeable with $F$ by $\Gamma$ then $\lrwd{(G,\gamma_G, \mu_G) \otimes_{\Gamma} (H,\gamma_H)} \leq p$.
\end{lem}

\begin{proof} Let $E:=(Y,Z,\mu_G, M,t)$ and $F:=(Y',Z', \mu_H,M',t)$ where $\mu_H=\emptyset$ and  $\mu_G=\{(v^1_1, v^1_2, t_1), (v^2_1, v^2_2, t_2), \ldots, (v^k_1, v^k_2, t_k) \}$. 
We define that 
\begin{enumerate}
\item $C_{v^i_1}$ is a function from $V(Y')$ to $\bF$ that maps $y\in V(Y')$ to $v^i_1\cdot \Gamma\cdot Y_2'(y)^t$,
\item $C_{v^i_2}$ is a function from $V(Z')$ to $\bF$ that maps $y\in V(Z')$ to $v^i_2\cdot \Gamma\cdot Z_2'(y)^t$,
\end{enumerate}
where the row indices of $Y'$ and $Z'$ are denoted by $V(Y')$ and $V(Z')$, respectively.

Let $\pi:V_G\cup V_H\to [t]$ such that
\begin{align*}
  \pi(x)&:=\begin{cases} L(x) & \textrm{if $x\in V_G$,}\\ L'(x) & \textrm{if $x\in V_H$}. \end{cases}
\end{align*}

By the assumption $\pi$ is an injective mapping, and let us take it as a linear layout of $K:=(G,\gamma_G, \mu_G) \otimes_{\Gamma} (H,\gamma_H)$. For each $j < t$ let $X_j:=\{x\in V_G\cup
V_H \mid \pi(x) \leq j\}$, and let 
  \begin{align*}
    A(j):=& \begin{pmatrix} Y_1(j)\cdot M(j) \cdot Z_1(j)^t & & Y_2(j)\cdot \Gamma \cdot {Z'}_2(j)^t \\ \\ \left(Z_2(j)\cdot \Gamma \cdot {Y'}_2(j)^t\right)^t & & N \end{pmatrix}.
  \end{align*}
 where $N=({Y'}_1(j) \cdot M'(j) \cdot {Z'}_1(j)^t)*(C_{v^1_1}, C_{v^1_2}, t_1)*(C_{v^2_1}, C_{v^2_2}, t_2)* \cdots *(C_{v^k_1}, C_{v^k_2}, t_k)$.
  By the definition of $p$-mergeability $\rk(A(j))\leq p$. Now by the definition of linear $s$-profiles of $s$-labelled graphs and Theorem \ref{thm:algebra-lrw} we have that $M_K[X_j,\overline{X_j}]$
  is obtained from $A(j)$ by copying rows and columns, \ie, $\rk(M_K[X_j,\overline{X_j}]) \leq p$. Hence, each cut $(X_j,\overline{X_j})$ of $\pi$ has rank at most $p$, \ie, $\lrwd{K}\leq p$.
\end{proof}

The following proves that $\lesssim^p$ respects $L$. 

\begin{prop}\label{prop:respectL} Let $(G,\gamma_G, \mu_G)$, $(G',\gamma_{G'}, \mu_{G'})$ be two boundaried $(s, \sigma, \bF^{*})$-graphs, and let $(H,\gamma_H)$ be an $(s, \sigma, \bF^{*})$-graph. Let $\Gamma$ be a matrix. If $\lrwd{(G,\gamma_G, \mu_G)
    \otimes_{\Gamma} (H,\gamma_H)} \leq p$ and $(G',\gamma_{G'}, \mu_{G'}) \lesssim^p (G,\gamma_G, \mu_G)$, then $\lrwd{(G',\gamma_{G'}, \mu_{G'})
    \otimes_{\Gamma} (H,\gamma_H)} \leq p$.
\end{prop}

\begin{proof} Let $\pi:=x_1x_2\ldots x_t$ be a linear layout of $(G,\gamma_G, \mu_G) \otimes_{\Gamma} (H,\gamma_H)$ of width at most $p$. Let $L_G:V_G\to [t]$ be such that $L_G(x):=L(x)$, and let us
  similarly define $L_H$. From Theorem \ref{thm:algebra-lrw} one can construct linear encodings $(N_G,P_G,M_G,L_G,t)$ of width $p_G$ and $(N_H,P_H,M_H,L_H,t)$ of width $p_H$ of $(G,\gamma_G, \mu_G)$ and
  $(H,\gamma_H)$ respectively. Let us denote by $E_G$ and $E_H$ the $(G,\gamma_G, \mu_G)$ and $(H,\gamma_H)$ profiles of $(N_G,P_G,M_G,L_G,t)$ and $(N_H,P_H,M_H,L_H,t)$ respectively. Clearly, from their
  definitions, $E_G$ is $p$-mergeable with $E_H$ by $\Gamma$.

  Since $(G',\gamma_{G'}, \mu_{G'})\lesssim^p (G,\gamma_G, \mu_G)$ there is a linear $s$-profile $E_{G'}$ of $(G',\gamma_{G'}, \mu_{G'})$ such that $E_{G'}\leq_D^p E_G$. Hence, there are subdivisions $E'_G$ and $E'_{G'}$ of
  $E_G$ and $E_{G'}$ respectively such that $E'_{G'}\leq_{DD}^p E'_{G}$. Moreover, if $(N_G,P_G,M_G,L_G,\ell)$ and $(N_{G'},P_{G'},M_{G'},L_{G'},\ell)$ are respectively linear encodings of $(G,\gamma_G, \mu_G)$ and
  $(G',\gamma_{G'}, \mu_{G'})$ associated with $E'_{G}$ and $E'_{G'}$ respectively, then for each $1\leq i \leq t$ we have that $|L_{G'}^{-1}(i)|\leq |L_G^{-1}(i)|$. Let $E_G^{(1)}:=E_G, E_G^{(2)},
  \ldots, E_G^{(l)}:=E'_{G}$ be a sequence of linear $s$-profiles such that $E_G^{(r+1)}$ is a subdivision of $E_G^{(r)}$ at $i_r$ for $1\leq r \leq l-1$. Let us denote by
  $(N_G^{(r)},P_G^{(r)},M_G^{(r)},L_G^{(r)},t+r)$ the linear encoding of $G$ associated with $E_G^{(r)}$. One can subsequently define a sequence of linear $s$-profiles $E_H^{(1)}:=E_H,\ldots,
  E_H^{(l)}$ such that $E_H^{(r+1)}$ is a subdivision of $E_H^{(r)}$ at $i_r$, for $1\leq r \leq l-1$, and $E_G^{(r)}$ is $p$-mergeable with  $E_H^{(r)}$ for $1\leq r\leq l$.  One can moreover
  associate with every $E_H^{(r)}$ a linear encoding $(N_H^{(r)},P_H^{(r)},M_H^{(r)},L_H^{(r)},t+r)$ of $H$ such that for each $1\leq i \leq t+r$ at most one of $V_G\cup V_H$ is mapped into $i$ by
  $L_G^{(r)}\cup L_H^{(r)}$.  

  One notices that $E'_{G'}$ and $E_H^{(l)}$ are linear $s$-profiles of $(G',\gamma_{G'}, \mu_{G'})$ and $(H,\gamma_H)$ respectively, and furthermore $(G',\gamma_{G'}, \mu_{G'})$ is $p$-mergeable with $(H,\gamma_H)$ from Fact~\ref{fact:dd-merge}. Since
  $|L_{G'}^{-1}(i)|\leq |L_G^{-1}(i)|$ and at most one vertex in $V_G\cup V_H$ is mapped into $i$ by $L_G\cup L_H^{(l)}$, one can conclude from Lemma \ref{lem:merge-lrw} that $\lrwd{(G',\gamma_{G'}, \mu_{G'})
    \otimes_{\Gamma} (H,\gamma_H)} \leq p$.
\end{proof}

We can now prove that $\lesssim^p$ is a pmo. 

\begin{lem}\label{lem:pmo-induced} Let $(G,\gamma_G, \mu_G)$ be an $s$-labelled graph. For every induced subgraph $(H,\gamma_H, \mu_H)$  of $(G,\gamma_G, \mu_G)$ we have that $(H,\gamma_H, \mu_H)\lesssim^p (G,\gamma_G, \mu_G)$. 
\end{lem}

\begin{proof} Let $E:=(Y,Z,\mu, M,t)$ be a linear $s$-profile of $(G,\gamma_G, \mu_G)$ and let $(N,P,M,L,t)$ be a linear encoding of $G$ of width $k$ such that $E$ is its $(G,\gamma_G, \mu_G)$ profile. 
	Suppose $(H,\gamma_H, \mu_H)$ is an induced subgraph of $(G,\gamma_G, \mu_G)$, and let $L'$ be the restriction of $L$ to $V_H$. 
	By the definition of induced subgraph, $\mu_H=\mu_G$.
	From Theorem \ref{thm:algebra-lrw} one can deduce from $(N,P,M,L,t)$ a linear encoding
  $(N',P',M',L',t)$ of $(H,\gamma_H, \mu_H)$ of width at most $k$ such that $N'$, $P'$ and $M'$ are sub-matrices of $N$, $P$ and $M$ respectively. Now, let $E':=(Y',Z',\mu_H,M',t)$ be the $(H,\gamma_H, \mu_H)$ profile
  of $(N',P',M',L',t)$. From the definition of $E'$ the matrices $Y'$, $Z'$ and $M'$ are sub-matrices of
  $Y$, $Z$ and $M$ respectively. Therefore, $E'\leq_{DD}^p E$
  since for each $i$ we have, by construction of $E'$, that $\rk(A_{E',\cD}(i))\leq \rk(A_{E,\cD}(i))$. 
\end{proof}


\begin{lem}\label{lem:pmo-lc} Let $(G,\gamma_G, \mu_G)$ be a boundaried $(s, \sigma, \bF^{*})$-graph and let $x,y\in V_G$ such that $M_G[x,y]=t$. If $(G',\gamma, \mu)$ is a pivot complementation at $xy$ of $(G,\gamma_G, \mu_G)$, then $(G',\gamma, \mu)\lesssim^p (G,\gamma_G, \mu_G)$. 
\end{lem}
\begin{proof}
 By the definition of pivot complementation we know that $G'=G\wedge xy$, $\mu:=\mu_G \Delta_{\bF} \{(\gamma(x), \gamma(y), t)\}$, and 
 \begin{align*}
  \gamma(z)=\begin{cases} 
  (1/\sigma(t))\cdot \gamma_G(y)  & \textrm{$z=x$},\\
  (\sigma(1)/t)\cdot \gamma_G(x)  & \textrm{$z=y$},\\
  \gamma_G(z) - (M_G[z,x]/\sigma(t))\cdot \gamma_G(y) - (M_G[z,y]/t)\cdot \gamma_G(x)  & \textrm{otherwise}. 
  \end{cases} 
\end{align*}
 Let $\mu_G=\{(v^1_1, v^1_2, t_1), (v^2_1, v^2_2, t_2), \ldots, (v^k_1, v^k_2, t_k) \}$.
 
  Let $E:=(Y,Z,\mu_G, M,t)$ be a linear $s$-profile of $(G,\gamma_G,\mu_G)$ and let $(N,P,M,L,t)$ be the linear encoding of $G$ such that $E$ is its $(G,\gamma_G, \mu_G)$ profile. From Theorem
  \ref{thm:algebra-lrw} one can deduce a linear encoding $(N',P',M',L,t)$ of $G\wedge xy$, of same width as $(N,P,M,L,t)$, such that for each $i\leq t$ ,
  $N'(i)\cdot M'(i)\cdot P'(i)^t=M_{G\wedge vw}[V(N(i)), V(P(i))]$ where 
  $V(N(i))$, $V(P(i))$ are the sets of row indices of $N(i)$ and $P(i)$, respectively.
  Let $E':=(Y',Z',\mu, M',t)$ be the linear $s$-profile of $(G',\gamma, \mu)$ associated with
  $(N',P',M',L,t)$. From the construction of $E'$ it remains to prove that $E'\leq_{DD}^p E$.


 Let $i\leq t$, and let $\cD:=(\Gamma, R, P:=(P_1,P_2),
  Q:=(Q_1,Q_2))$ be an $(s,p)$-matrix tuple and 
  let $V_P$ and $V_Q$ be the row indices of $P$ and $Q$, respectively.
  We define that for each $1\le i\le k$,
\begin{enumerate}
\item $C_{v^i_1}$ is a function from $V_P$ to $\bF$ that maps $y\in V_P$ to $v^i_1\cdot \Gamma\cdot P_2(y)^t$,
\item $C_{v^i_2}$ is a function from $V_Q$ to $\bF$ that maps $y\in V_Q$ to $v^i_2\cdot \Gamma\cdot Q_2(y)^t$.
\end{enumerate}
  Also, let us define 
  \begin{align*}
    A& :=Y_1(i)\cdot M(i)\cdot Z_1(i)^t & B&:=Y_2(i)\cdot \Gamma \cdot Q_2^t & 
    C&:=(Z_2(i)\cdot \Gamma\cdot {P_2}^t)^t \\
    A'&:=Y'_1(i)\cdot M(i)\cdot {Z'_1(i)}^t & B'&:=Y'_2(i)\cdot \Gamma \cdot {Q_2}^t & 
    C'&:=(Z'_2(i)\cdot \Gamma\cdot
      {P_2}^t)^t 
      \end{align*}
      and 
      \begin{align*}
      D&:=(P_1\cdot R\cdot Q_1^t)*(C_{v^1_1}, C_{v^1_2}, t_1)*(C_{v^2_1}, C_{v^2_2}, t_2)* \cdots *(C_{v^k_1}, C_{v^k_2}, t_k)	 \\
      D'&:=D* (\gamma(x), \gamma(y), t).
      \end{align*}
    Then by definition,  
\begin{align*}
  A_{E,\cD}(i)&:=\begin{pmatrix} A & B \\ C & D \end{pmatrix} & \quad\textrm{and}\quad 
A_{E',\cD}(i)&:=\begin{pmatrix} A' & B' \\ C' & D' \end{pmatrix}.
\end{align*}

	Now we claim that $\rk(A_{E,D}(i)) = \rk(A_{E',D}(i))$. 
	We choose $S, T\subseteq V_G$ that are indices of $N(i)$, $P(i)$, respectively, such that $M_G[S,T]=A$. 
	If $M_G[x,V_G\setminus \{x,y\}]=M_G[y,V_G\setminus \{x,y\}]$ and $\gamma_G(x)=\gamma_G(y)$, then pivoting $xy$ will not change anything.
	So, we may assume that at least one of the equalities does not hold, and therefore we can assume without loss of generality that $\{x, y\}\subseteq S\cup T$.
	Let $H$ be the graph with $V_H=S\cup T\cup V_P\cup V_Q$ such that
	$M_{H}[S\cup T]=M_{G}[S\cup T]$, $M_{H}[V_P\cup V_Q]=D$, and
	$M_{H}[U, V_R]=U_2(i)\cdot \Gamma \cdot R_2^t$ for each $U\in \{Y,Z\}$ and $R\in \{P,Q\}$.
	From the construction of $H$, it is clear that $M_H[S\cup V_P, T\cup V_Q]=A_{E,D}(i)$.	By Lemma~\ref{lem:preserverank}, it is enough to show that $M_{H\wedge xy}[S\cup V_P, T\cup V_Q]=A_{E',D}(i)$.

	In the encoding $(N',P',M',L,t)$, $N'(i)$ and $P'(i)$ satisfy that $N'(i)\cdot M'(i)\cdot P'(i)^t=M_{G\wedge xy}[V(N(i)), V(P(i))]$.
	Since $Y_1(i)$ and $Z_1(i)$ consist of vectors in $N'(i)$ and $P'(i)$ respectively, $A'=Y_1'(i)\cdot M'(i)\cdot Z_1'(i)^t= M_{G\wedge xy}[S,T]=M_{H\wedge xy}[S, T]$.
	
	Now we want to observe the submatrices $B$ and $C$. 
Note that by the change of $\gamma$ from $\gamma_G$, $Y_2'(i)$ and $Z_2'(i)$ are transformed from $Y_2(i)$ and $Z_2(i)$, respectively. From this fact,
	we can observe that for $z\in S$ and $q\in V_Q$,
\begin{align*}
 & B'[z,q]=(Y_2'(i)\cdot \Gamma \cdot Q_2^t)[z,q]\\&=\begin{cases} 
  (1/\sigma(t))\cdot M_{H}[y,q]  & \textrm{$z=x$},\\
  (\sigma(1)/t)\cdot M_{H}[x,q]  & \textrm{$z=y$},\\
  M_H[z,q] - (M_G[z,x]/\sigma(t))\cdot M_H[y,q] - (M_G[z,y]/t)\cdot M_H[x,q]  & \textrm{otherwise}
  \end{cases} \\
  &=M_{H\wedge xy}[z,q].
\end{align*}
This implies that $M_{H\wedge xy}[S, V_Q]=B'$, and similarly, we can easily check that $M_{H\wedge xy}[T, V_P]=(C')^t$.

It remains to show that $D'=M_{H\wedge xy}[V_P, V_Q]$.
	From the definition of pivot complementation,
	for each $p\in V_P, q\in V_Q$, we know that
	\begin{align*}M_{H\wedge xy}[p, q] = M_H[p, q] - (M_H[p,x]\cdot M_H[y, q])/\sigma(t) -  (M_H[p,y]\cdot M_H[x, q])/t 
	\end{align*}
For each $z\in \{x,y\}$ and $r\in \{p,q\}$,
since $\gamma(z)\cdot \Gamma\cdot P_2(r)^t=M_H[z,r]$, 
we have that $C_{\gamma(z)}(r)=M_H[z,r]$.
Therefore, 
\begin{align*}&M_H[p, q] - (M_H[p,x]\cdot M_H[y, q])/\sigma(t) -  (M_H[p,y]\cdot M_H[x, q])/t \\
 	& =M_{H}[p, q] - (\sigma(C_{\gamma(x)}(p)) \cdot C_{\gamma(y)}(q))/\sigma(t) -  (\sigma(C_{\gamma(y)}(p))\cdot C_{\gamma(x)}(q))/t \\
 	&= (M_{H}[V_P, V_Q] * (C_{\gamma(x)}, C_{\gamma(y)}, t))[p,q].
	\end{align*}
	Finally, we have that 
	\[D'=D* (C_{\gamma(x)}, C_{\gamma(y)}, t)=M_{H}[V_P, V_Q]* (C_{\gamma(x)}, C_{\gamma(y)}, t)=M_{H\wedge vw}[V_P, V_Q].\]
	Altogether, we prove that $M_{H\wedge xy}[S\cup V_P, T\cup V_Q]=A_{E',D}(i)$, and therefore \[\rk(A_{E,D}(i)) = \rk(A_{E',D}(i)).\]
   Since $i$ is arbitrary, we conclude that $E'\leq_{DD}^p E$. 
\end{proof}

From Lemmas \ref{lem:pmo-induced} and \ref{lem:pmo-lc} we can deduce that $\lesssim^p$ is a pmo. 

\begin{prop}\label{prop:pmo} Let $(G,\gamma_G,\mu_G)$ be a boundaried $s$-labelled graph. For every $(H,\gamma_H, \mu_H)$ that is a pivot-minor of $(G,\gamma_G, \mu_G)$, we have that $(H,\gamma_H, \mu_H)\lesssim^p (G,\gamma_G, \mu_G)$.
\end{prop}

\subsection{Bounding the Length of a Pseudo-Minor Order}

The goal now is to bound the size of the chains of the pmo $\lesssim^p$. The method consists in defining an equivalence relation and proving that in each equivalence class there is a member of bounded
size. 

A linear $s$-profile $(Y,Z,\mu, M,t)$ is \emph{redundant} if there are indices $i$ and $j$ such that $|j-i-1|\geq 1$ and for each $\min\{i,j\} \leq \ell \leq \max\{i,j\}$
\begin{enumerate}
\item[(R1)] $Rest(Y_2(i))=Rest(Y_2(\ell))=Rest(Y_2(j))$ and $Rest(Z_2(i))=Rest(Z_2(\ell))=Rest(Z_2(j))$, 
\item[(R2)] For each  $(s,p)$-matrix tuple $\cD:=(\Gamma, N, P:=(P_1,P_2), Q:=(Q_1,Q_2))$ we have \begin{align*} \rk (A_{E,\cD}(i)) \leq \rk (A_{E,\cD}(\ell)) \leq \rk (A_{E,\cD}(j)).\end{align*}
\end{enumerate}

We call the pair $(i,j)$ a \emph{$p$-redundant pair}. Given $E:=(Y,Z,\mu, M,t)$ and a $p$-redundant pair $(i,j)$, the \emph{$p$-shortcut of $E$ at $(i,j)$} is the linear $s$-profile $E':=(Y',Z',\mu, M',t-(j-i-1))$
where for each $s\leq i$, 
	\[(Y'(s),Z'(s),M'(s)) :=(Y(s),Z(s),M(s))\] 
	and for each $s>i$, 
	\begin{align*}(Y'(s),&Z'(s),M'(s)):= \\&(Y(s+(j-i-1)), Z(s+(j-i-1)),
    M(s+(j-i-1)))).\end{align*}

\begin{prop}\label{prop:equiv-redundant} Let $E$ be a linear $s$-profile and $E'$ a $p$-shortcut of $E$ at $(i,j)$. Then $E\simeq_D^p E'$. 
\end{prop}

\begin{proof} We can assume without loss of generality that $j>i$. If we subdivide $E'$ at $i$ for $(j-i-1)$ times and denote it by $F'$ we clearly have $F'\leq_{DD}^p E$, \ie, $E'\leq_D^p E$.  Similarly,
  if we subdivide $E'$ at $j$ for $(j-i-1)$ times and denote it by $F''$ we also have $E\leq_{DD}^p F''$, \ie, $E\leq_D^p E'$. Therefore, $E'\simeq_D^p E$.
\end{proof}

A linear $s$-profile is called \emph{non-$p$-redundant} if it does not contain a $p$-redundant pair. Now for each equivalence class \wrt $\simeq_D^p$ we can only consider non-$p$-redundant ones thanks
to Proposition \ref{prop:equiv-redundant}.

A non-$p$-redundant linear $s$-profile $E:=(Y,Z,\mu, M,t)$ is called a \emph{$p$-homogenous linear $s$-profile} if for each $i\ne j$ we have 
\begin{align*} (Rest(Y_2(i)),Rest(Z_2(i))) &=
(Rest(Y_2(j)),Rest(Z_2(j))).\end{align*} 

For a $p$-homogenous linear $s$-profile $(Y,Z,\mu,M,t)$ and $(s,p)$-matrix tuple $\cD:=(\Gamma, N, P:=(P_1,P_2), Q:=(Q_1,Q_2))$ the index $i\in [t]$ is called an \emph{extreme index \wrt $\cD$} if either
\big($\rk(A_{E,\cD}(i))> \rk(A_{E,\cD}(i'))$ for all $i'\ne i$\big), or \big($\rk(A_{E,\cD}(i)) < \rk(A_{E,\cD}(i'))$ for all $i'\ne i$\big).

\begin{lem}\label{lem:exists-extreme} Let $\cD:=(\Gamma, N, P:=(P_1,P_2), Q:=(Q_1,Q_2))$ be an $(s,p)$-matrix tuple. Every $p$-homogenous linear $s$-profile $(Y,Z,\mu, M,t)$ with $t\geq 2$ has an
  extreme index \wrt $\cD$.
\end{lem}

\begin{proof} The proof is the same as in \cite[Lemma 4.8]{Lagergren98}. Let us define $Inf:=\min\{\rk(A_{E,\cD}(s))\mid s\in [t]\}$ and $Sup:=\max\{\rk(A_{E,\cD}(s))\mid s\in [t]\}$.  Let $Min:=\{s\in [t]\mid
  \rk(A_{E,\cD}(s)) = Inf\}$ and let $Max:=\{s\in [t]\mid \rk(A_{E,\cD}(s)) = Sup\}$. If the lemma is false then since neither $Min$ nor $Max$ is empty, we have clearly that $|Min|,|Max|\geq 2$. Let us enumerate the
  indices of $Min\cup Max$ as $i_1<i_2<\cdots < i_p$.

  Assume first there is $1\leq j \leq p$ such that $i_j$ and $i_{j+1}$ are both in the same set, say $Min$. If $i_{j+2}$ is in $Max$, then the pair $(i_j,i_{j+2})$ is $p$-redundant since for each $i_j\leq
  \ell \leq i_{j+2}$ we have $\rk(A_{E,\cD}(i_j)) \leq \rk(A_{E,\cD}(\ell)) \leq \rk(A_{E,\cD}(i_{j+2}))$ and $|i_{j+2}-i_j|\geq 1$ because $i_j\leq i_{j+1} \leq i_{j+2}$. Similarly, if $i_{j-1}$ is in $Max$, then also for
  similar reasons $(i_{j+1},i_{j-1})$ is a $p$-redundant pair.  In both cases we contradict the non-$p$-redundancy of $E$. The case when $i_j$ and $i_{j+1}$ are in $Max$ is analogous.

  From above if $i_1\in Min$, then $i_4\in Max$, and similarly if $i_1\in Max$, then $i_4\in Min$. In the first case $(i_1,i_4)$ is a $p$-redundant pair, and in the second case $(i_4,i_1)$ is a $p$-redundant
  pair.  We again contradict the non-$p$-redundancy of $E$. We can thus conclude that one of $Min$ and $Max$ has exactly one element.
\end{proof}

\begin{lem}\label{lem:neighbor-extreme} Let $\cD:=(\Gamma, N, P:=(P_1,P_2), Q:=(Q_1,Q_2))$ be an $(s,p)$-matrix tuple. Let $E:=(Y,Z,\mu, M,t)$ with $t\geq 2$ be a $p$-homogenous linear $s$-profile such
  that $1$ is an extreme index \wrt $\cD$. Then $2$ is also an extreme index \wrt $\cD$.
\end{lem}

\begin{proof}The proof is similar to \cite[Lemma 4.9]{Lagergren98}. Suppose that $\rk(A_{E,\cD}(1)) < \rk(A_{E,\cD}(i))$ for all $1<i\leq t$. We claim that $\rk(A_{E,\cD}(2)) > \rk(A_{E,\cD}(i))$ for all $i\ne
  2$. Let $Sup:=\max\{\rk(A_{E,\cD}(j))\mid j\in [t]\}$ and let $Max:=\{j\in [t]\mid \rk(A_{E,\cD}(j)) = Max\}$. Let $i_k$ be the greatest index in $Max$. If $i_k\ne 2$, then $(1,i_k)$ is a $p$-redundant pair,
  contradicting that $E$ is non-$p$-redundant.  We conclude that $i_k=2$, which implies that $\rk(A_{E,\cD}(2))> \rk(A_{E,\cD}(i))$ for all $i\ne 2$. The case $\rk(A_{E,\cD}(1)) > \rk(A_{E,\cD}(i))$ for all
  $1<i\leq t$ is similar. 
\end{proof}

\begin{lem}\label{lem:size-source-homogenous-extreme} Let $\cD:=(\Gamma, N, P:=(P_1,P_2), Q:=(Q_1,Q_2))$ be an $(s,p)$-matrix tuple and let $E:=(Y,Z,\mu, M,t)$ be a $p$-homogenous linear $s$-profile such
  that $\max\{\rk(A_{E,\cD}(i)) \mid i \in [t]\} \leq p'$. If $1$ is an extreme index \wrt $\cD$ then $t\leq p'+1$.
\end{lem}

\begin{proof} By induction on $p'$. If $p'=0$, then if $1$ is an extreme index we have $\rk (A_{E,\cD}(1)) = 0$, and then we cannot have another index with value strictly greater and smaller than
  $0$. Assume now that $p'>0$. If $\rk(A_{E,\cD}(1)) > \rk(A_{E,\cD}(i))$ for all $1<i\leq t$, then let $E':=(Y',Z',M',t-1)$ which is the restriction of $E$ to $\{2,\ldots, t\}$ and if $\rk(A_{E,\cD}(1)) <
  \rk(A_{E,\cD}(i))$ for all $1<i\leq t$, then let $E':=(Y',Z',M',t-1)$ which is the restriction of $E$ to $\{1,3,\ldots, t\}$. In both cases $E'$ is also a $p$-homogenous linear $s$-profile, and in
  the first case $2$ is an extreme index \wrt $\cD$ (by Lemma \ref{lem:neighbor-extreme}) and in the second case $1$ is still an extreme index \wrt $\cD$. Moreover, $\max\{\rk(A_{E',\cD}(i)) \mid i \in
  [t-1]\} \leq p'-1$ since in the first case $\max\{\rk(A_{E,\cD}(i)) \mid 2\leq i \leq t\} \leq p'-1$, and in the second case by Lemma \ref{lem:neighbor-extreme} $\rk(A_{E,\cD}(2)) > \rk(A_{E,\cD}(i))$ for all
  $i\ne 2$ . By inductive hypothesis $t-1\leq p'$, meaning $t\leq p'+1$.
\end{proof}

\begin{lem}\label{lem:size-source-homogenous} Let $\cD:=(\Gamma, N, P:=(P_1,P_2), Q:=(Q_1,Q_2))$ be an $(s,p)$-matrix tuple and let $E:=(Y,Z,\mu, M,t)$ be a $p$-homogenous linear $s$-profile such that
  $\max\{\rk(A_{E,\cD}(i)) \mid i \in [t]\} \leq p'$. Then $t\leq 2p'+1$.
\end{lem}

\begin{proof} By Lemma \ref{lem:exists-extreme} $E$ has an extreme index $i$ \wrt $\cD$. Let $E_1$ be $E$ restricted to the first $i$ indicies, and let $E_2$ be $E$ restricted to the last indices
  $\{i,\ldots,t\}$. In both cases $i$ is an extreme index \wrt $\cD$. By Lemma \ref{lem:size-source-homogenous-extreme} we know that (1) $t-i+1 \leq p'+1$.  And with Fact \ref{fact:dual} we know that
  $\max \{\rk(A_{E_1^d,\cD^d}(j))\mid j\in [i]\} = \max\{\rk(A_{E,\cD}(j))\mid j\in [i]\} \leq p'$, and then again by Lemma \ref{lem:size-source-homogenous-extreme} we know that (2) $i\leq p'+1$.  Using (1)
  and (2) we can conclude that $t\leq 2p'+1$.
\end{proof}


\begin{lem}\label{lem:size-non-redundant} Let $\cD:=(\Gamma, N, P:=(P_1,P_2), Q:=(Q_1,Q_2))$ be an $(s,p)$-matrix tuple and let $E:=(Y,Z,\mu, M,t)$ be a non-$p$-redundant linear $s$-profile such that
  $\max\{\rk(A_{E,\cD}(i)) \mid i \in [t]\} \leq p'$. Then $t\leq |\bF|^{2s}\cdot (2p'+1)$.
\end{lem}

\begin{proof} Since in a linear $s$-profile we have the matrices $Rest(Y_2)$ that are increasing and $Rest(Z_2)$ that are decreasing, both \wrt sub-matrix inclusion, the number of indices $i\ne j$ such that
  $Rest(Y_2(i))\ne Rest(Y_2(j))$ is bounded by $|\bF|^s$ and the number of different indices $i'\ne j'$ such that $Rest(Z_2(i'))\ne Rest(Z_2(j'))$ is also bounded by $|\bF|^s$, \ie, the number of indices $i\ne j$ such that
  $(Rest(Y_2(i)),Rest(Z_2(i)))\ne (Rest(Y_2(j)),Rest(Z_2(j)))$ is bounded by $|\bF|^{2s}$. So, the number of maximal $p$-homogenous linear $s$-profiles that are restrictions of $E$ is bounded by $|\bF|^{2s}$.  By Lemma \ref{lem:size-non-redundant}
  the size of each such $p$-homogenous linear $s$-profile  is bounded by $2p'+1$, \ie, $t\leq |\bF|^{2s}\cdot (2p'+1)$. 
\end{proof}


\begin{prop}\label{prop:p-length} The pmo $\lesssim^p$ for linear rank-width at most $p$ on the linear $s$-profiles has length at most 
 $(2p+1)\cdot |\bF|^{p^2+s^2+2s + |\bF|^{p+s}\cdot (4p+2s) + |\bF|^{2s+1}}$.
\end{prop}

\begin{proof} Since in each equivalence class of $\simeq_D^p$ one can find a non-$p$-redundant linear $s$-profile, it is enough to bound the number of non-$p$-redundant linear $s$-profiles. Since we are
  interested in linear $s$-profiles of $p$-width at most $p$, by Observation \ref{obs:rk} we are only interested  in non-$p$-redundant linear $s$-profiles such that each $Y(i), Z(i)$ is of order
  $|\bF|^{p+s}\times p$ and each $M(i)$ of order $p\times p$. By Lemma \ref{lem:size-non-redundant} for each  $\cD:=(\Gamma, N, P:=(P_1,P_2), Q:=(Q_1,Q_2))$ and such linear $s$-profile
  $E:=(Y,Z,\mu,M,t)$, we have that $t\leq |\bF|^{2s}\cdot (2p+1)$. 
  Also, since each boundary $\mu$ is a $\bF$-multiset that consists of triples $v_1, v_2\in \bF^s$, $t\in \bF$,
  the maximum of different elements in $\mu$ is $|\bF|^{2s+1}$.
  Then, for each $\cD:=(\Gamma, N, P:=(P_1,P_2), Q:=(Q_1,Q_2))$ the number of such linear $s$-profiles is bounded by 
  \[|\bF|^{2s}\cdot (2p+1)\cdot
  |\bF|^{2\cdot |\bF|^{p+s}\cdot p}\cdot |\bF|^{|\bF|^{2s+1}}\cdot |\bF|^{p^2}.\] Since for each such $\cD$ the matrix $\Gamma$ is of order at most $s\times s$, $P_1$, $Q_1$ are of order $|\bF|^{p+s}\times p$, $P_2$ and $Q_2$ are of order  $|\bF|^{p+s}\times s$, so the
  number of such $\cD$s is bounded by 
  \[|\bF|^{s^2}\cdot |\bF|^{2\cdot |\bF|^{p+s}\cdot s}\cdot |\bF|^{2\cdot |\bF|^{p+s}\cdot p}.\] So, the number of such linear $s$-profiles is bounded by 
\[   (2p+1)\cdot |\bF|^{p^2+s^2+2s + |\bF|^{p+s}\cdot (4p+2s) + |\bF|^{2s+1}.} \qedhere\]
\end{proof}

\begin{proof}[Proof of Main Theorem (Theorem \ref{thm:main})] By Proposition \ref{prop:p-length} linear rank-width at most $p$ has $s$-length at most 
\[ (2p+1)\cdot |\bF|^{p^2+s^2+2s + |\bF|^{p+s}\cdot (4p+2s) + |\bF|^{2s+1}.} \] 
Since any obstruction for linear rank-width $p$ has linear
  rank-width at most $p+1$ we can conclude using Theorem \ref{thm:upper-bound} that the number of vertices of $G$ is bounded by 
\[  \left(|\bF|^{|\bF|^{(2p+1)+\log_{|\bF|} (6p+1)}}\right)^{O(p+1)} =
|\bF|^{|\bF|^{(2p+1)+2\log_{|\bF|} (O(p+1))}} \leq |\bF|^{|\bF|^{O(p)}}.
\qedhere \]  
\end{proof}


\section{Obstructions for Path-Width of Representable Matroids}\label{sec:matroid}

As a corollary of Theorem~\ref{thm:main}, 
we can obtain an upper bound on the size of obstructions for $\bF$-representable matroids of bounded path-width, where $\bF$ is a finite field.
For connecting our result with matroids, 
we establish a direct relation between $\bF$-representable matroids and skew-symmetric bipartite $\bF^*$-graphs, which is similar to the relation between binary matroids and bipartite graphs.

We recall the necessary materials about matroids. Let $\bF$ be a finite field. We refer to \cite{Oxley12} for our matroid terminology. There exist several characterisations of matroids, but we will define and use only one of
them.

A pair $\cM=(E_{\cM},\cI_{\cM})$ is called a \emph{matroid} if $E_{\cM}$, called \emph{ground set} of $\cM$, is a finite set and $\cI_{\cM}$, called \emph{independent sets} of $\cM$, is a nonempty collection of
subsets of $E_{\cM}$ satisfying the following conditions:
\begin{description}
\item[(I1)] if $I\in \cI_{\cM}$ and $J\subseteq I$, then $J\in \cI_{\cM}$,
\item[(I2)] if $I,J\in \cI_{\cM}$ and $|I|<|J|$, then $I\cup \{z\}\in \cI_{\cM}$ for
  some $z\in J\backslash I$.
\end{description}

A maximal independent set in $\cM$ is called a \emph{base} of $\cM$. It is well-known that, if $B_1$ and $B_2$ are bases of $\cM$, then $|B_1|=|B_2|$. Two matroids $\cM$ and $\mathcal{N}$ are \emph{isomorphic} if there
exists a bijection $h:E_{\cM}\to E_{\mathcal{N}}$ such that $I\in \cI_{\cM}$ if and only if $h(I)\in \cI_{\mathcal{N}}$.

If $\cM$ is a matroid and $X$ a subset of $E_{\cM}$, we let $(X,\{I\subseteq X\mid I \in \cI_{\cM}\})$ be the matroid denoted by $\restriction{\cM}{X}$.  The size of a base of $\restriction{\cM}{X}$ is called the
\emph{rank} of $X$ and the \emph{rank function} of $\cM$ is the function $r_{\cM}:2^{E_{\cM}}\to \bN$ that maps every $X\subseteq E_{\cM}$ to its rank. The rank of $E_{\cM}$ is called the rank of $\cM$. It is
well-known that the rank function is submodular.
  
Let $A$ be a matrix over a field $\bF$ and let $E$ be the column labels of $A$. Let $\cI$ be the collection of all those subsets $I$ of $E$ such that the columns of $A$ with index in $I$ are linearly
independent. Then $\cM(A):=(E,\cI)$ is a matroid. Any matroid isomorphic to $\cM(A)$ for some matrix $A$ is said \emph{representable over $\bF$} and $A$ is called a \emph{representation of $\cM$ over
  $\bF$}. 

We now define the \emph{matroid minor} notion. Let $\cM$ be a matroid. The \emph{dual} of $\cM$, denoted by $\cM^*$, is the matroid $(I\subseteq E_{\cM},\{E_{\cM}\setminus B\mid B\in \cI_{\cM}$ and $B$ is a base$\})$. For
$X\subseteq E_{\cM}$, we let $\cM\setminus X$ be the matroid $\restriction{\cM}{(E_{\cM}\setminus X)}$ called the \emph{deletion of $X$ from $\cM$}, and we let $\cM/X$ be the matroid $(\cM^*\setminus X)^*$ called the
\emph{contraction of $X$ from $\cM$}. A matroid $\mathcal{N}$ is a \emph{minor} of a matroid $\cM$ if it is isomorphic to $\cM\setminus X/Y$ for disjoint subsets $X$ and $Y$ of $E_{\cM}$. Observe that
$\bF$-representable matroids are closed under minors.

We finish these preliminaries with the notion of \emph{path-width} of matroids. If $\cM$ is a matroid, we let $\lambda_{\cM}$, called the \emph{connectivity function of $\cM$}, be such that for
every subset $X$ of $E_{\cM},\ \lambda_{\cM}(X) = r_{\cM}(X) + r_{\cM}(E\setminus X) -r_{\cM}(E_{\cM}) +1$. It is well-known that the function $\lambda_{\cM}$ is symmetric and submodular. The \emph{path-width
  of $\cM$}, denoted by $\lbwd{\cM}$, is the linear width of $\lambda_{\cM}$. The path-width is sometimes called linear width or linear branch-width.  A matroid is said \emph{connected} if $\lambda_{\cM}(X)\geq
1$ for every subset $X$ of $E_{\cM}$. The (inclusionwise) maximal subsets $X$ of $E_{\cM}$ such that $\lambda_{\cM}(X)=1$ are called \emph{connected components} of $X$.  One can easily check that the path-width of a matroid is equal to the maximum path-width of its connected components (concatenate optimal linear layouts of its connected components). The path-width of a matroid does not increase when taking a minor.

\begin{prop}[\cite{Oxley12}] \label{prop:2.1.1} Let $\cM$ be a matroid and $\mathcal{N}$ a minor of $\cM$. Then, $\lambda_{\mathcal{N}}(X)\leq \lambda_{\cM}(X)$ for every subset $X$ of $E_N$. Therefore, $\lbwd{\mathcal{N}} \leq \lbwd{\cM}$.
\end{prop}

The main result of this section is the following.

\begin{thm}\label{thm:m-obstruction}  Let $\bF$ be a finite field. 
  If $\cM$ is an $\bF$-representable matroid and is an obstruction for path-width at most $p$, then $|E_{\cM}|$ is at most doubly exponential in $O(p)$.
\end{thm}

Let $G$ be a skew-symmetric bipartite $\bF^*$-graph with a bipartition $(A, B)$.
We define $\cM_{\bF}(G, A, B)$ as the $\bF$-representable matroid represented by the $A\times V$ matrix $(I_A \quad M_G[A, B])$ where $I_A$ is the $A\times A$ identity matrix.
If $\cM=\cM_{\bF}(G, A, B)$, then we call $G$ a \emph{fundamental graph} of $\cM$.

We can relate the rank function of a $\bF$-representable matroid  $\mathcal{M}$ with the cut-rank function of its fundamental graph.

\begin{prop}\label{prop:matroidrank}
Let $G$ be a skew-symmetric bipartite $\bF^*$-graph with a bipartition
$(A, B)$ and let $\cM: = \cM_{\bF}(G, A,B)$. 
For every $X\subseteq V_G$, $\ucutrk_G(X)=\lambda_{\cM}(X)-1$.
Thus, $\lrwd{G}=\lbwd{\cM} -1$.
\end{prop}
\begin{proof}
We first observe that
\begin{align*}
M_G[X, V_G\setminus X]=\begin{pmatrix} 0 & M_G[X\cap A,(V_G\setminus X)\cap B] \\ M_G[X\cap B, (V_G\setminus X)\cap A] & 0 \end{pmatrix}.
\end{align*}

From the definition of the rank function of a matroid, we have that 
\begin{align*}
\lambda_{\cM}(X)-1 &=\rk_{\cM}(X) + \rk_{\cM}(V_G\setminus X) - \rk_{\cM}(V_G) \\
	&= \rk \begin{pmatrix} 0 & M_G[(V_G\setminus X)\cap A,X\cap B] \\ I_{X\cap A} & M_G[X\cap A, X\cap B] \end{pmatrix} \\
	&+ \rk \begin{pmatrix} I_{(V_G\setminus X)\cap A} & M_G[(V_G\setminus X)\cap A,(V_G\setminus X)\cap B] \\ 0 & M_G[X\cap A, (V_G\setminus X)\cap B] \end{pmatrix} - |A| \\
	&=\rk (M_G[(V_G\setminus X)\cap A,X\cap B]) +\rk (M_G[X\cap A, (V_G\setminus X)\cap B]) \\
	&=\ucutrk_G(X). \qedhere
\end{align*}
\end{proof}


Minors of $\bF$-representable matroids are related with pivot-minors of their fundamental graphs.
From the definition of pivot complementation, 
it is not hard to check that every pivot-minor of a skew-symmetric bipartite $\bF^*$-graph is a skew-symmetric bipartite $\bF^*$-graph.
The next proposition is similar to \cite[Proposition 3.3]{Oum05}, but for the second statement, we need to be careful that $M_G$ is not symmetric.

\begin{prop}\label{prop:minorfg}
Let $G$ be a skew-symmetric bipartite $\bF^*$-graph with a bipartition
$(A, B)$ and let $\cM: = \cM_{\bF}(G, A,B)$. 
Then the following are satisfied.
\begin{enumerate}
\item $\cM_{\bF}(G, B,A)=\cM^{*}$.
\item For $x\in A$ and $y\in B$, $\cM_{\bF}(G\wedge xy, A\Delta \{x,y\},B\Delta\{x,y\})=\cM$ and $\cM_{\bF}(G\wedge yx, A\Delta \{x,y\},B\Delta\{x,y\})=\cM$.
\item For $x\in V_G$, 
$\cM_{\bF}(G\setminus x, A\setminus \{x\}, B\setminus \{x\})
=\begin{cases} \cM/ x & \textrm{if $x\in A$,}\\ \cM\setminus x & \textrm{if $x\in B$}. \end{cases}$.
\end{enumerate}
\end{prop}

\begin{proof}
(1) Since $M_G$ is skew-symmetric, $\cM^{*}$ has the following as a representation  (see for instance \cite{Oxley12})
\[(-M_G[A, B]^t \quad I_B)=(M_G[B,A] \quad I_B),\] 
which implies that  $\cM_{\bF}(G, B,A)=\cM^{*}$.

(2) Note that $\cM$ has a representation $M=(I_A \quad M_G[A, B])$ and any row operations do not change the associated matroid. Also, replacing a column with a multiple of this column does not change the associated matroid.

We first consider $G\wedge xy$.
The representation matrix $M$ can be written as
\begin{align*}
\begin{pmatrix} 
1 & 0 & M_G[x,y] & M_G[x, B\setminus \{y\}] \\ 
0 & I_{A\setminus \{x\}} & M_G[A\setminus \{x\},y] & M_G[A\setminus \{x\}, B\setminus \{y\}]\end{pmatrix}.
\end{align*}
We first multiply the first row by $(-M_G[z,y]/M_G[x,y])$ and add it to the row indexed by $z$ for each $z\in A\setminus \{x\}$. Then by multiplying the first row by $-1/M_G[x,y]$, and the third column by $-1$, we obtain a new matrix
\begin{align*}
M'=\begin{pmatrix} 
-1/M_G[x,y] & 0 & 1 & -M_G[x, B\setminus \{y\}]/M_G[x,y] \\ 
-M_G[A\setminus \{x\},y]/M_G[x,y] & I_{A\setminus \{x\}} & 0 & \widetilde{M}\end{pmatrix},
\end{align*}
where \[\widetilde{M}=
M_G[A\setminus \{x\}, B\setminus \{y\}]- (M_G[A\setminus \{x\}, y]\cdot M_G[x,B\setminus \{y\}])/M_G[x,y].\]
By permuting the two columns indexed by $x$, $y$ and replacing the index of the first row by $y$ on $M'$,
we have the matrix
\[\begin{pmatrix} 
I_{A\Delta \{x,y\}} \quad M_{G\wedge xy}[A\Delta \{x,y\}, B\Delta\{x,y\}] \end{pmatrix}.\] Therefore, $\cM_{\bF}(G\wedge xy, A\Delta \{x,y\},B\Delta\{x,y\})=\cM$.

If we do pivot complementation at $yx$, then 
from the definition of pivot complementation, 
$M_{G\wedge yx}[A\setminus \{x\}, B\setminus \{y\}]=M_{G\wedge xy}[A\setminus \{x\}, B\setminus \{y\}]$, and we have
\begin{align*}
&M_{G\wedge yx}[y,x]=\frac{-1}{M_G[x,y]}=M_{G\wedge xy}[y,x],\\ 
&M_{G\wedge yx}[y, B\setminus \{y\}]=\frac{M_G[x, B\setminus \{y\}]}{M_G[x,y]}=-M_{G\wedge xy}[y, B\setminus \{y\}],\\
&M_{G\wedge yx}[A\setminus \{x\}, x]=
\frac{M_G[A\setminus \{x\}, y]}{M_G[x,y]}=-M_{G\wedge xy}[A\setminus \{x\}, y].\end{align*}
In this case, we first multiply the first column of $M$ by $-1$. Also,  
we multiply the first row by $(M_G[z,y]/M_G[x,y])$ and add it to the row indexed by $z$ for each $z\in A\setminus \{x\}$, and multiply the first row by $1/M_G[x,y]$.
Then we have a matrix 
\begin{align*}
M''=\begin{pmatrix} 
-1/M_G[x,y] & 0 & 1 & M_G[x, B\setminus \{y\}]/M_G[x,y] \\ 
M_G[A\setminus \{x\},y]/M_G[x,y] & I_{A\setminus \{x\}} & 0 & \widetilde{M}\end{pmatrix}.
\end{align*}
By permuting the two columns indexed by $x$, $y$ and replacing the index of the first row by $y$ on $M''$,
we obtain the matrix  
   \[(I_{A\Delta \{x,y\}} \quad M_{G\wedge yx}[A\Delta \{x,y\}, B\Delta \{x,y\}]),\]
and we conclude that $\cM_{\bF}(G\wedge yx, A\Delta \{x,y\},B\Delta\{x,y\})=\cM$.

(3) If $x\in B$, then $M'=(I_{A}, M_G[A, B\setminus \{x\}])$ is a representation of $\cM\setminus x$. Thus, 
$\cM_{\bF}(G\setminus x, A\setminus \{x\}, B\setminus \{x\})=\cM\setminus x$.
If $x\in A$, $(I_B\quad M_G[B, A])$ is a representation of $\cM^*$ and $(I_B \quad M_G[B, A\setminus \{x\}])$ is a representation of $\cM^*\setminus x$.
Therefore, 
$(I_A\quad M_G[A\setminus \{x\}, B])$ is a representation of $(\cM^*\setminus x)^*=\cM/x$, as required.
\end{proof}

\begin{prop}\label{prop:representationofminor}
Let $G$ be a skew-symmetric bipartite $\bF^*$-graph with a bipartition $(A, B)$ and let $\cM:=\cM_{\bF}(G, A, B)$. Let $x\in V_G$.
If $x$ has no neighbor in $G$, then $\cM/x=\cM\setminus x=\cM_{\bF}(G\setminus x, A\setminus \{x\}, B\setminus \{x\})$.
Suppose $x$ has a neighbor $y$ in $G$.
Then \begin{align*}
\cM\setminus x=
\begin{cases}\cM_{\bF}(G\wedge xy\setminus x, A\Delta \{x,y\},B\Delta \{x,y\}\setminus \{x\}) & \text{if } x\in A \\ 
\cM_{\bF}(G\setminus x, A\setminus \{x\},B\setminus \{x\}) & \text{otherwise},\end{cases}
\end{align*}
and, 
\begin{align*}
\cM/ x=
\begin{cases}\cM_{\bF}(G\wedge xy\setminus x, A\Delta \{x,y\}\setminus \{x\},B\Delta \{x,y\}) & \text{if } x\in B  \\ 
\cM_{\bF}(G\setminus x, A\setminus \{x\},B\setminus \{x\}) & \text{otherwise},\end{cases}
\end{align*}
\end{prop}
\begin{proof}
It is easy to check that if $x$ has no neighbor in $G$, then $\cM/x=\cM\setminus x=\cM_{\bF}(G\setminus x, A\setminus \{x\}, B\setminus \{x\})$.
Suppose that $x$ has a neighbor $y$ in $G$.
If $x\in B$, then $\cM\setminus x$ corresponds to removing the column indexed by $x$. So, $\cM=\cM_{\bF}(G\setminus x, A\setminus \{x\},B\setminus \{x\})$. If $x\in A$, then by Proposition~\ref{prop:minorfg}, 
 $\cM=\cM_{\bF}(G\wedge xy, A\Delta \{x,y\},B\Delta \{x,y\})$ and 
  $\cM\setminus x=\cM_{\bF}(G\wedge xy\setminus x, A\Delta \{x,y\},B\Delta \{x,y\}\setminus \{x\})$.

Now we consider the contraction operation.
If $x\in A$, then by Proposition~\ref{prop:minorfg}, 
$\cM/x=\cM_{\bF}(G\setminus x, A\setminus \{x\},B\setminus \{x\})$.
If $x\in B$, then 
 $\cM=\cM_{\bF}(G\wedge xy, A\Delta \{x,y\},B\Delta \{x,y\})$ and 
$\cM/x=\cM_{\bF}(G\wedge xy\setminus x, A\Delta \{x,y\}\setminus \{x\},B\Delta \{x,y\})$.
\end{proof}

\begin{cor}\label{cor:samefg}
For each $i\in \{1,2\}$, let $G_i$ be a skew-symmetric bipartite $\bF^*$-graph with a bipartition $(A_i, B_i)$ such that
$\cM_{\bF}(G_1, A_1, B_1)=\cM_{\bF}(G_2, A_2, B_2)$.
Then $G_1$ is pivot equivalent to $G_2$.
\end{cor}
\begin{proof}
Let $\cM:=\cM_{\bF}(G_1, A_1, B_1)$.
	We prove it by induction on $|A_1\Delta A_2|$.
	If $A_1=A_2$, then $B_1=B_2$ and 
	the edges between $A_i$ and $B_i$ should be same because $\cM_{\bF}(G_1, A_1, B_1)=\cM_{\bF}(G_2, A_2, B_2)$.
	Therefore $G_1=G_2$.
	
	Suppose that $|A_1\Delta A_2|\neq 0$. 
	Note that $A_1$ and $A_2$ are basis of $\cM$.
	If $x\in A_1\setminus A_2$,
	then by the second condition (I2) for being a matroid,
	there exists $y\in A_2\setminus A_1$
	such that $A_1\setminus \{x\}\cup \{y\}$ is a base.
	By Proposition~\ref{prop:minorfg}, 
	$\cM_{\bF}(G_1\wedge xy, A_1\Delta \{x,y\}, B_1\Delta \{x,y\})=\cM$, 
	and since $|(A_1\Delta \{x,y\})\Delta A_2|=|A_1\Delta A_2|-1$, by induction hypothesis,
	$G_1\wedge xy$ is pivot equivalent to $G_2$.
	Therefore, $G_1$ is pivot equivalent to $G_2$.
\end{proof}

\begin{prop}\label{prop:matroid-graph} 
\begin{enumerate}
\item Let $\cM_1$, $\cM_2$ be $\bF$-representable matroids, and let $G_1$ and $G_2$ be fundamental graphs of $\cM_1$ and $\cM_2$ respectively. If $\cM_1$ is a minor of $\cM_2$, then $G_1$ is a pivot-minor of $G_2$.
\item Let $G$ be a skew-symmetric bipartite $\bF^*$-graph with a bipartition
$(A, B)$. 
If $G'$ is a pivot-minor of $G$, 
then there is a bipartition $(A', B')$ of $G'$ such
that $\cM_{\bF}(G', A',B')$ is a minor of $\cM_{\bF}(G, A,B)$.
\end{enumerate}
\end{prop}

\begin{proof}
  (1) By Proposition~\ref{prop:representationofminor}, there exists a fundamental graph $G_1'$ of $\cM_1$ such that $G_1'$ can be obtained from $G_2$ by applying pivot complementations and vertex
  deletions. Since both $G_1$ and $G_1'$ are fundamental graphs of $\cM_1$, by Corollary~\ref{cor:samefg}, $G_1$ is pivot equivalent to $G_1'$.  Therefore, $G_1$ is a pivot-minor of $G_2$.

(2) This is clear from Proposition~\ref{prop:minorfg}.
\end{proof}

We can deduce this easy lemma.
\begin{lem}\label{lem:ominor-pminor} 
Let $k$ be a positive integer. 
Let $\cM$ be an $\bF$-representable matroid and $G$ a fundamental graph of $\cM$. 
Then $\cM$ is a minor obstruction for path-width at most $k$ if and only if $G$ is a pivot-minor obstruction for linear rank-width at most $k+1$.
\end{lem}

\begin{proof} Assume $\cM$ is an obstruction for path-width at most $k$. By Proposition~\ref{prop:matroidrank}, $\lbwd\cM=k+1$ and $\lrwd G=k+2$. If $G$ is not a pivot-minor obstruction for linear rank-width $k+1$, then there is a proper pivot-minor $G'$ that has linear rank-width $k+2$.  By Proposition
  \ref{prop:matroid-graph} there exists a bipartition $(A', B')$ of $G'$ such that $\cM_{\bF}(G', A', B')$ is a proper minor of $\cM$ and has path-width $k+1$. It contradicts to that $\cM$ is a minor obstruction for path-width at most $k$.
  
 Suppose $G$ is a pivot-minor obstruction for linear rank-width at most $k+1$ and $\cM$ is not an obstruction for path-width at most $k$. Then there is a proper minor $\mathcal{N}$ of $\cM$ that has path-width $k+1$. Then by Proposition \ref{prop:matroid-graph} 
 a fundamental graph of $\mathcal{N}$ is a proper pivot-minor of $G$ and has path-width $k+1$, contradicting that $G$ is a pivot-minor obstruction.
\end{proof}

As a corollary we have the following.
\begin{cor}\label{cor:m-obstruction} If $\cM$ is an $\bF$-representable matroid and a minor obstruction for path-width at most $p$, then $|E_{\cM}|$ is at most doubly exponential in $O(p)$.
\end{cor}

\begin{proof} Let $G$ be a fundamental graph of $\cM$. By Lemma \ref{lem:ominor-pminor} $G$ is a pivot-minor obstruction for linear rank-width at most $p+1$. By Theorem \ref{thm:main} we have that $|E_{\cM}| = |V(G)|$ is doubly exponential in
  $O(p)$. 
\end{proof}

\section{Concluding Remarks}\label{sec:5}
We present an $|\bF|^{|\bF|^{\mathcal{O}(p)}}$ upper bound on the size of pivot-minor obstructions for $\sigma$-symmetric $\bF^*$-graphs of linear rank-width at most $p$, by exploring the ideas of
Lagergren~\cite{Lagergren98} to bound the size of minor obstructions for graphs of bounded path-width.  For undirected graphs, it implies that vertex-minor obstructions for linear rank-width at most
$p$ have size at most $2^{2^{\mathcal{O}(p)}}$, which answers the open question explicitly posted by Jeong, Kwon, and Oum~\cite{JKO2014}.  As mentioned in the paper, in theory, one can enumerate all
graphs up to that bound and construct the list of forbidden vertex-minors. This gives the first explicit FPT algorithm on $p$ to decide whether a given graph has linear rank-width at most $p$.

Also, as a corollary of the main result, we have an $|\bF|^{|\bF|^{\mathcal{O}(p)}}$ upper bound on the size of minor obstructions for $\bF$-representable matroids of path-width at most $p$.  By the
same argument, we have a fixed parameter algorithm to decide whether a given $\bF$-representable matroid with a representation has path-width at most $p$.  
\medskip

We conclude with some questions. 
\begin{enumerate}
\item The bound on the size of the $\sigma$-symmetric $\bF^*$-graph obstructions for rank-width at most $p$ does not depend on the size of the field, while our bound depends on it. Can we obtain a
  bound not depending on the size of the field?

\item Our bound on the size of the vertex-minor obstructions for bounded linear rank-width is doubly exponential, however the best lower bound is singly exponential \cite{JKO2014}. Is our bound optimal? 
\end{enumerate}


  \end{document}